\documentclass[12pt,reqno]{amsart}
\usepackage{amsfonts}
\usepackage{amsmath,amssymb,amsthm, mathrsfs,amsopn}
\usepackage{a4wide}
\allowdisplaybreaks
\numberwithin{equation}{section}
 \usepackage[pagewise]{lineno}\linenumbers\nolinenumbers
\usepackage{graphicx}
\usepackage{epstopdf}
\usepackage{float}
\usepackage{stfloats}
\usepackage{subfigure}
\usepackage{soul}
\usepackage{url}
\usepackage{color}
\usepackage[colorlinks, linkcolor=blue, citecolor=blue]{hyperref}

\newtheorem{theorem}{Theorem}[section]

\newtheorem{lemma}[theorem]{Lemma}

\theoremstyle{definition}

\newtheorem{remark}[theorem]{Remark}

\newcommand{\R}{\mathbb{R}}

\newcommand{\eeq}{\end{equation}}
\newcommand{\beq}{\begin{equation*}}

\begin{document}
	
	\title[  prescribed the mass solution]{
Normalized solutions on large smooth domains to the  Schr\"{o}dinger equations with potential  and combined nonlinearities: The Sobolev critical case }

	\thanks{The research was supported by National Science Foundation of China (****)}

\author{ Xiaolu Lin}
	\address{ Xiaolu Lin,~School of Mathematics and Statistics, Central China Normal University, Wuhan 430079, P. R. China}
	\email{ xllin@ccnu.edu.cn}
	\author{Yanjun Liu}
	\address{ Yanjun Liu,~School of Mathematical Sciences, Chongqing Normal University, Chongqing 401331, P. R. China}
	\email{liuyj@mail.nankai.edu.cn}

	\author{Zongyan Lv}
	\address{ Zongyan Lv,~ Center for Mathematical Sciences, Wuhan University of Technology, Wuhan 430070, P. R. China	}
	\email{zongyanlv0535@163.com}
	
	\date{}

\maketitle

\begin{abstract}
In this paper,  we consider the existence and multiplicity of prescribed mass
solutions to the following nonlinear Schr\"{o}dinger equations with mixed nonlinearities:
\begin{equation*}
    \begin{cases}
        -\Delta u+V(x)u+\lambda u=|u|^{2^*-2}u+\beta|u|^{p-2}u \\
        \|u\|_2^2=\int|u|^2\mathrm{d}x=\alpha,
    \end{cases}
\end{equation*}
both on large bounded smooth star-shaped domain $\Omega\subset\mathbb{R}^{N}$ and on $\mathbb{R}^{N}$, where $N \geq 3$, $2<p<2+\frac{4}{N}$, $2^*=\frac{2 N}{N-2}$  is the critical Sobolev  exponent and $V(x)$ is the potential. The standard approach based on the Pohozaev identity to obtain normalized solutions is invalid as the presence of potential $V(x)$. Besides,  Our study can be regarded as a Sobolev critical case complement of Bartsch-Qi-Zou (Math Ann 390, 4813--4859, 2024), which has addressed an open problem raised in Bartsch et al. (Commun Partial Differ Equ 46(9):1729--1756, 2021).

		\vskip 0.2cm
		\noindent{\bf MR(2010) Subject Classification:} {35J60; 35J20; 35R25}
	\end{abstract}

\section{Introduction}

	The aim of this paper is to  study the existence and multiplicity of solutions for the following nonlinear Schr\"{o}dinger equation with the critical Sobolev growth:
		\begin{equation}\label{main-eq}
			\begin{cases}
				-\Delta u+V(x)u+\lambda u=|u|^{2^*-2}u+\beta|u|^{p-2}u\quad&\text{in}\ \Omega,\\
				u\in H_0^1(\Omega),\quad\int_{\Omega}|u|^2dx=\alpha,
			\end{cases}
		\end{equation}
	where $\Omega \subset \mathbb{R}^N$ is either a bounded smooth star-shaped domain or all of $\mathbb{R}^N$, $N \geq 3$, $2<p<2+\frac{4}{N}$, $2^*=\frac{2 N}{N-2}$  is the critical Sobolev  exponent and $V(x)$ is the potential. The mass $\alpha>0$ and the parameter $\beta \in \mathbb{R}$ are prescribed. Here  the frequency $\lambda$  appears as a Lagrange multiplier.
		
 It is well known that Eq \eqref{main-eq} comes from the study of standing waves for the nonlinear Schr\"{o}dinger
		equation (NLS) with combined power nonlinearities:
		\begin{equation}\label{schro-eq}
				i\frac{\partial\Phi}{\partial t}-V(x)\Phi+\Delta\Phi+\beta|\Phi|^{p-2}\Phi+|\Phi|^{2^*-2}\Phi=0, \quad (t,x)\in\mathbb{R}\times\mathbb{R}^N,
		\end{equation}
		The Schr\"{o}dinger equation is a fundamental equation in quantum mechanics, which can be used to describe many physical phenomena, for example, the nonlinear optical problems and the Bose-Einstein condensates, see, e.g. \cite{AEMW95,ESY10}.
		The ansatz $\Phi(x,t)=e^{i\lambda t}u(x)$ for standing waves solutions leads to the equation
		\begin{equation}\label{nls-el-eq}
			-\Delta u+V(x)u+\lambda u=f(|u|)u,\quad\text{in}\ \mathbb{R}^{N}.
		\end{equation}
		
			In the equation \eqref{nls-el-eq}, if $\lambda $ is given, we call it fixed frequency problem.
For a fixed frequency $\lambda \in \R  $,	the existence and multiplicity of solutions to \eqref{nls-el-eq} has been investigated in the last two decades by many authors.
The literature in this direction is huge and we do not even make an attempt to summarize it here (see
		e.g. \cite{AIIK19,AlJ22,AlT23,CP09,HLW24} for a survey on almost classical results).

		Recall that the important feature of  \eqref{schro-eq} is that  $L^2$-norm of $\Phi(\cdot,u)$  is  conserved in time, therefore it is natural to consider \eqref{nls-el-eq} with the constraint $\int_{\Omega}|u|^2dx=\alpha$. (see \cite{CL})
		In this case, the mass $\alpha > 0$  is prescribed, while the frequency $\lambda$ is unknown and  comes out as a Lagrange multiplier.  	The existence and properties of these normalized solutions recently has attracted the attention 	of many researchers.

		In the autonomous case, i.e. the potential $V$ is a constant, the NLS equation on the whole space $\mathbb{R}^N$ with combined  power nonlinearity  has attracted a lot of attention since the classical paper Tao-Visan-Zhang \cite{TVZ07} appeared.
		For example, about the situation  $V(x)=0$:
		if the combined power nonlinearities is purely $L^2$-subcritical, i.e. $2<p<q<2+\frac{4}{N}$, then the energy is bounded from below on $S_\alpha$. Thus, for every $\alpha,\beta>0$, a ground state can be found as global minimizers of the energy functional constrained in $S_\alpha$, see \cite{Shi17}.
		In the purely $L^2$-supercritical case, i.e. $2+\frac{4}{N}<p<q$, the main  difficulty is that the energy is unbounded from below on $S_\alpha$; however,  Jeanjean \cite{Jean} showed that a normalized groundstate does exist by applying a smart compactness argument, Pohozaev identity and the mountain pass lemma. In \cite{BS17}, Bartsch and Soave further established that the Pohozaev manifold $\{u\in S_\alpha:P(u)=0\}$ is a natural constraint. Moreover, the conditions in \cite{Jean,BS17} can be weaken by reference \cite{Jeanlu20,BM21}.
		Recently,  Soave \cite{Soa20} studied what happens when the combined power nonlinearities  are
		of mixed type, that is $2<p<2+\frac{4}{N}<q<2^*$. Under different assumptions on $\alpha>0,\beta\in\mathbb{R}$, Soave proved several existence and stability/instability by decomposing the Nehari-Pohozaev manifold in a subtle way.
        In particular, for $q=2^*$,
		Soave \cite{Soac20}  further given the existence of ground states, and the existence of mountain-pass solution for $N\ge5$.
		In \cite{JL22}, Jeanjean-Le further proved that, when $N\ge4$, there also
		exist standing waves which are not ground states and are located at a mountain-pass level of the Energy functional. These solutions are unstable by blow-up in finite time.
		For $N=3$, Wei-Wu \cite{WW22} obtianed the existence of solutions of mountain-pass type.
		For more information on the existence of solutions for mixed nonlinear terms of Schr\"{o}dinger equation and systems, we recommend reference \cite{BLZ23,QZ23}.

		We would like to mention here that some paper involved in  dealing with the potential $V(x)$ is  non-constant.   For the case $V\le0$, $V(x)\to0$ as $|x|\to\infty$(decaying potentials), and $f$ is mass subcritical, Ikoma-Miyamoto \cite{IM20} proved that  $e(\alpha)$ is attained for $\alpha>\alpha_0$ and $e(\alpha)$ is not attained for $0<\alpha<\alpha_0$. For the case $V(x)\ge0$, $V(x)\to0$ as $|x|\to\infty$, and allow that the potential has singularities.
        Bartsch-Molle-Rizzi-Verzini \cite{BMRV21} concerned with the existence of normalized solutions based on a new min-max argument. The recent papers \cite{DZ,BHG2,ZZ} and references	therein for new contributions.


		There are only a few approaches and results on the study of normalized solutions in bounded domains \cite{CDCE,BHSG,BHG1,BHG3,DG,DGY,SZ,WJ1}, which dealt with the the autonomous case. In view of the inherent characteristics of the problem with prescribed mass, the so-called Pohozaev manifold is not available when working in bounded domains.  However,  it is worth to point that  the method of these papers above don not work for non-constant $V$.
        Recently, Bartsch-Qi-Zou \cite{BQZ24} considered the existence and properties of normalized solutions with a combination of  mass subcritical and  mass supercritical, i.e. $f(|u|)u=|u|^{q-2}+\beta|u|^{p-2}u$, with $2<p<2+\frac{4}{N}<q<2^*$.

In spired by \cite{BQZ24},  as naturally
expected, we are going to consider the Sobolev critical case, i.e. $q=2^*$.   The study of the
convergence of Palais-Smale sequences becomes more complicated as the presence of the Sobolev critical term in \eqref{main-eq}.
            Therefore, we have to recover the compactness, which makes the problem more difficult.
             By subtle energy estimates, we can obtain the existence of normalized solution to \eqref{main-eq} on large bounded smooth star-shaped domains $\Omega_r$ and further obtain the asymptotic behavior (as the radius $r$ tends to infinity), i.e. the existence of normalized solution in whole space.
Specially, the present paper seems to be the first
result for normalized solutions to the  Schr\"{o}dinger equations with potential  and combined nonlinearities of  Sobolev critical case  in $\Omega$ which is large bounded domain even expand to the whole space.

		Throughout of this paper, we will use the following notations: Let $m_+:=\max\{m,0\}$, $m_-:=\min\{m,0\}$ with $m\in\mathbb{R}$. For $\Omega\subset\mathbb{R}^{N}$, $r>0$ we set
		$
		\Omega_r=\Big\{rx\in\mathbb{R}^{N}:x\in\Omega\Big\}
		$
		and
		\[
		S_{r,\alpha}:=S_{\alpha}\cap H_0^1(\mathbb{R}^{N})=\Big\{u\in H_0^1(\Omega_r):\|u\|_{L^2(\Omega_r)}^2=\alpha\Big\}.
		\]
		without loss of generality,  we assume that $\Omega\subset\mathbb{R}^{N}$ is a bounded smooth domain with $0\in\Omega$ and star-shaped with respect to 0.  Let $S$ the optimal constant of the Sobolev embedding from $D^{1,2}(\mathbb{R}^{N})$ to $L^{2^*}(\mathbb{R}^{N})$ see \cite{Aub76}.
		Before stating our main results, we  state our basic assumptions on the potential,
		\begin{description}
			\item[$(V_0)$] $V\in C(\mathbb{R}^{N})\cap L^{\frac{N}{2}}(\mathbb{R}^{N})$ is bounded and $\|V_-\|_{\frac{N}{2}}<S$.
		\end{description}
		For some results we will require that $V$ is $C^1$ and define
		$\tilde{V}:\mathbb{R}^{N}\to\mathbb{R}$ by $\tilde{V}(x)=\nabla V(x)\cdot x$.
		In our first result, we consider the case $\beta>0$.

\begin{theorem}\label{beta>0-e<0-Omega}
Suppose that $(V_0)$ and $\beta>0$ hold, and set
\[
\alpha_V=\Bigg(\frac{p}{\beta C_{N,p}}\bigg(1-\frac{4-N(p-2)}{2(2^*-p)}\bigg)
\bigg(\frac{(1-\|V_-\|_{\frac{N}{2}}S^{-1})}{2}\bigg)^{\frac{N(2^*-p)}{2(2^*-2)}}
\bigg(\frac{(4-N(p-2))S^{2^*/2}2^*}{2(2^*-p)}\bigg)^{\frac{4-N(p-2)}{2(2^*-2)}}
\Bigg)^{\frac{4}{2p-N(p-2)}}.
\]
Then the following hold for $0<\alpha<\alpha_V$:
\begin{description}
  \item[(i)]There exists $r_\alpha>0$ such that \eqref{main-eq} on $\Omega_r$ with $r>r_\alpha$ has a local minimum type solution $(\lambda_{r,\alpha},u_{r,\alpha})$ with $u_{r,\alpha}>0$ in $\Omega_r$ and negative energy $E_V(u_{r,\alpha})<0$.
  \item[(ii)] There exists $C_\alpha>0$ such that
  \[
  \underset{r\to\infty}{\limsup}\,\underset{x\in\Omega_r}{\max}\,u_{r,\alpha}(x)<C_\alpha,\quad
  \underset{r\to\infty}{\liminf}\,\lambda_{r,\alpha}>0.
  \]
\end{description}
\end{theorem}

\begin{theorem}\label{beta>0-e>0-Omega}
Suppose that $(V_0)$ and $\beta>0$ hold, and $V(x)\in C^1$ and $\tilde{V}$ is bounded. Set
\[
\tilde{\alpha}_V=\bigg(\frac{(1-\|V_-\|_{\frac{N}{2}}S^{-1})
S^{\frac{N}{N-2}}N^2}{(N-2)^2(A_p+1)}\bigg)^{\frac{N}{2}}
\bigg(\frac{A_p}{S(N-2)^2}\bigg)^{\frac{4}{2p-N(p-2)}},
\]
where
\[
A_p=\frac{32}{(N-2)(p-2)(4-N(p-2))}.
\]
Then there is $\alpha_0>0$ such that the following hold for $0<\alpha<\alpha_1:=\min\{\tilde{\alpha}_V,\alpha_0\}$:
\begin{description}
  \item[(i)]There exists $\tilde{r}_\alpha>0$ such that \eqref{main-eq} on $\Omega_r$ with $r>\tilde{r}_\alpha$ possess a mountain pass type solution $(\lambda_{r,\alpha},u_{r,\alpha})$ with $u_{r,\alpha}>0$ in $\Omega_r$ and positive energy $E_V(u_{r,\alpha})>0$. Moreover, there exists $C_\alpha>0$ such that
  \[
  \underset{r\to\infty}{\limsup}\,\underset{x\in\Omega_r}{\max}\,u_{r,\alpha}(x)<C_\alpha.
  \]
  \item[(ii)] There exists $0<\bar{\alpha}\le \alpha_1$ such that
  \[
  \underset{r\to\infty}{\liminf}\,\lambda_{r,\alpha}>0\quad\text{for any}\ 0<\alpha\le\bar{\alpha}.
  \]
\end{description}
\end{theorem}

\begin{remark}
    It is worth mentioning that in the proof of Theorem \ref{beta>0-e>0-Omega}(ii), we do not need the condition $\|\tilde{V}_+\|_{S}\le 2S$ in \cite{BQZ24}.
\end{remark}

\begin{theorem}\label{betale0-e>0-Omega}
Suppose that $(V_0)$ and $\beta\le0$ hold, and $V(x)\in C^1$ and $\tilde{V}$ is bounded. Then the following hold:
\begin{description}
  \item[(i)] There are $r_\alpha,\alpha_2>0$ such that for $\alpha\in\left(0,\alpha_2\right)$, \eqref{main-eq} on $\Omega_r$ with $r>r_\alpha$ possess a mountain pass type solution $(\lambda_{r,\alpha},u_{r,\alpha})$ with  the following properties: $u_{r,\alpha}>0$ in $\Omega_r$ and positive energy $E_V(u_{r,\alpha})>0$. Moreover, there exists $C_\alpha>0$ such that
      \[
      \underset{r\to\infty}{\limsup}\,\underset{x\in\Omega_r}{\max}\,u_{r,\alpha}(x)<C_\alpha.
      \]
  \item[(ii)] There exists $\tilde{\alpha}\in\left(0,\alpha_2\right)$ such that
  \[
  \underset{r\to\infty}{\liminf}\lambda_{r,\alpha}>0\quad\text{for any}\ 0<\alpha<\tilde{\alpha}.
  \]
\end{description}
\end{theorem}

The rest of this manuscript is organized as follows.
In Section 2, we  collect some notations and some preliminary results   which will be used in this paper.
In Section 3, we obtain the existence and properties of ground states with prescribed mass in large bounded smooth star-shaped domains for the case $\beta>0$.
Section 4 and 5 are devoted to the existence of mountain-pass type solutions for the cases $\beta>0$ and $\beta\le0$, respectively.

\section{Preliminaries.}
\setcounter{equation}{0}
\setcounter{theorem}{0}	

This section is devoted to collect some preliminary results which will be used in this paper.
Let us first introduce the Gagliardo-Nirenberg inequality, see \cite{Wei82}.
\begin{lemma}
For any $N\ge2$ and $p\in\left(2,2^*\right)$, there is a constant $C_{N,p}$ depending on
N and p such that
\[
\|u\|_s^s\le C_{N,s}\|u\|_2^{\frac{2s-N(s-2)}{2}}\|\nabla u\|_2^{\frac{N(s-2)}{2}}\quad\forall\ u\in H^1(\mathbb{R}^{N}),
\]
where $C_{N,s}$ be the best constant.
\end{lemma}

Next we recall the Monotonicity trick  \cite{BCJS10,CJS22}.
\begin{theorem}\label{mp-th}(Monotonicity trick)
Let $(E,\langle\cdot,\cdot\rangle)$ and $(H,(\cdot,\cdot))$ be two infinite-dimensional Hilbert spaces and assume there are continuous injections
\[
E\hookrightarrow H\hookrightarrow E'
\]
Let
\[
\|u\|^2=\langle u,u\rangle,\quad |u|^2=(u,u)\quad\text{for}\,u\in E,
\]
and
\[
S_\mu=\{u\in E:|u|^2=\mu\},\quad T_u S_\mu=\{v\in E:(u,v)=0\}\quad\text{for}\,\mu\in\left(0,+\infty\right).
\]
Let $I\subset\left(0,+\infty\right)$ be an interval and consider a family of $C^2$ functionals $\Phi_\rho: E\to\mathbb{R}$ of the form
\[
\Phi_\rho(u)=A(u)-\rho B(u)\quad\text{for}\,\rho\in I,
\]
with $B(u)\ge0$ for every $u\in E$, and
\[
A(u)\to+\infty\quad\text{or}\quad B(u)\to+\infty\quad\text{as}\,u\in E\,\text{and}\,\|u\|\to+\infty.
\]
Suppose moreover that $\Phi'_\rho$ and $\Phi''_\rho$ are H\"{o}lder continuous, $\tau\in\left(0,1\right]$, on bounded sets in the following sense: for every $R>0$ there exists $M=M(R)>0$ such that
\begin{equation}
\|\Phi'_\rho(u)-\Phi'_\rho(v)\|\le M\|u-v\|^\tau\quad \|\Phi''_\rho(u)-\Phi''_\rho(v)\|\le M\|u-v\|^\tau
\end{equation}
for every $u,v\in B(0,R)$. Finally, suppose that there exist $w_1,w_2\in S_\mu$ independent of $\rho$ such that
\[
c_\rho:=\underset{\gamma\in\Gamma}{\inf}\underset{t\in\left[0,1\right]}{\max}\Phi_\rho(\gamma(t))
>\max\{\Phi_\rho(w_1),\Phi_\rho(w_2)\}\quad\text{for all}\,\rho\in I,
\]
where
\[
\Gamma=\{\gamma\in C(\left[0,1\right],S_\mu):\gamma(0)=w_1,\gamma(1)=w_2\}.
\]
Then for almost every $\rho\in I$ there exists a sequence $\{u_n\}\subset S_\mu$ such that
\begin{description}
  \item[(i)] $\Phi_\rho(u_n)\to c_\rho$,
  \item[(ii)] $\Phi'_\rho|_{S_\mu}(u_n)\to0$,
  \item[(iii)] $\{u_n\}$ is bounded in $E$.
\end{description}
\end{theorem}

Let $U_\varepsilon$ be defined by
\[
U_\varepsilon(x)=\bigg(\frac{\varepsilon}{\varepsilon^2+|x|^2}\bigg)^{\frac{N-2}{2}},
\quad\varepsilon>0,
\]
where $U_\varepsilon(x)$ is the bubble in $\mathbb{R}^{N}$ centered in the origin, with concentration parameter $\varepsilon$.
Let also $\varphi\in C_0^\infty(\Omega,\left[0,1\right])$ be a nonnegative function such that $\varphi\equiv1$ on $B_\rho$, $\rho>0$ and define
$u_\varepsilon(x):=\varphi(x)U_\varepsilon(x)$.

\begin{lemma}\label{u-varepsilon-es}\cite[Lemma 7.1]{JL22}
We have for $N\ge3$, the following holds:
\begin{description}
  \item[(i)]
  \[
  \|\nabla u_\varepsilon\|_2^2=S^{\frac{N}{2}}+O(\varepsilon^{N-2}),\quad
  \|u_\varepsilon\|_{2^*}^{2^*}=S^{\frac{N}{2}}+O(\varepsilon^N),
  \]
  \item[(ii)] For some constant $K_4>0$ and,
  \[
  \|u_\varepsilon\|_p^p=
  \begin{cases}
  K_4\varepsilon^{N-\frac{N-2}{2}p}+o(\varepsilon^{N-\frac{N-2}{2}p})&\text{if}\,\frac{N}{N-2}<p<2^*,\\
  K_4\varepsilon^{\frac{N}{2}}|\text{ln}\varepsilon|+O(\varepsilon^{\frac{N}{2}})
  &\text{if}\,p=\frac{N}{N-2},\\
  K_4\varepsilon^{\frac{N-2}{2}p}+o(\varepsilon^{\frac{N-2}{2}p})&\text{if}\,1\le p<\frac{N}{N-2},
  \end{cases}
  \]
  where $K_4$ is a positive constant.
\end{description}
\end{lemma}

The energy functionals has different geometric structures due to the sign of $\beta$, which leads to need adopt different approaches to investigate the existence and multiplicity of normalized solutions in  $\Omega_r$.
 Consider
\begin{equation}\label{main-eq-omega}
\begin{cases}
-\Delta u+V(x)u+\lambda u=|u|^{2^*-2}u+\beta|u|^{p-2}u&\text{in}\ \Omega_r,\\
u\in H^{1}_{0}(\Omega_r),\ \int_{\Omega_r}|u|^2dx=\alpha.
\end{cases}
\end{equation}
where $N\ge3$, $2<p<2+\frac{4}{N}$, the mass $\alpha>0$ and the parameter $\beta\in\mathbb{R}$ are prescribed, and the frequency $\lambda$ is unknown.
The energy functional $E_r:H_0^1(\Omega_r)\to\mathbb{R}$ is defined by
\[
E_r(u)=\frac{1}{2}\int_{\Omega_r}|\nabla u|^2dx+\frac{1}{2}\int_{\Omega_r}V(x)u^2dx
-\frac{1}{2^*}\int_{\Omega_r}|u|^{2^*}dx-\frac{\beta}{p}\int_{\Omega_r}|u|^pdx,
\]
and the mass constraint manifold is defined by
\[
S_{r,\alpha}=\Big\{u\in H_0^1(\Omega_r):\|u\|_2^2=\alpha\Big\}.
\]

\section{Proof of Theorem \ref{beta>0-e<0-Omega}}
\setcounter{equation}{0}
\setcounter{theorem}{0}	
Suppose that the assumptions of Theorem \ref{beta>0-e<0-Omega} hold in this section. Let $0<\alpha<\alpha_V$ be fixed,
 to understand the geometry of the functional $E_r|_{S_{r,\alpha}}$, it is useful to
consider the function $h:\mathbb{R}^+\to\mathbb{R}$
\[
h(t):=\frac{1}{2}\bigg(1-\|V_-\|_{\frac{N}{2}}S^{-1}\bigg)t^2
-\frac{\beta C_{N,p}}{p}\alpha^{\frac{2p-N(p-2)}{4}}t^{\frac{N(p-2)}{2}}
-\frac{S^{-2^*/2}}{2^*}t^{2^*}.
\]
The role of the definition of $\alpha_V$ is clarified by the following lemma.
\begin{lemma}
Under the assumptions of Theorem \ref{beta>0-e<0-Omega}, the function $h$  has a global maximum at positive level. Moreover, there exist three positive constants  $R_1<T_\alpha<R_2$ such that
\[
h(R_1)=h(R_2)=0,\quad h(t)>0\ \text{for}\ R_1<t<R_2,\quad\text{and}\ h(T_\alpha)={\max}_{t\in\mathbb{R}^+}h(t)>0.
\]
\end{lemma}
\begin{proof}
 Since $\beta>0$ and $2<p<2+\frac{4}{N}$, we have that $h(0^+)=0^-$ and $h(+\infty)=-\infty$. For $t>0$,
\[
h(t)
=t^{\frac{N(p-2)}{2}}\bigg(\phi(t)-\frac{\beta C_{N,p}}{p}\alpha^{\frac{2p-N(p-2)}{4}}\bigg),
\]
where
\[
\phi(t)=\frac{1}{2}\bigg(1-\|V_-\|_{\frac{N}{2}}S^{-1}\bigg)t^{\frac{4-N(p-2)}{2}}
-\frac{S^{-2^*/2}}{2^*}t^{\frac{N(2^*-p)}{2}}.
\]
It is easy to verify that $\phi$ admits a unique maximum at
\[
\bar{t}=\bigg(\frac{(1-\|V_-\|_{\frac{N}{2}}S^{-1})(4-N(p-2))2^*S^{2^*/2}}{2N(2^*-p)}
\bigg)^{\frac{N-2}{4}}
\]
After manipulation and  the definition of $\alpha_V$, we get
\begin{equation}\label{phi-below-bdd}
\phi(\bar{t})>\frac{\beta C_{N,p}}{p}\alpha^{\frac{2p-N(p-2)}{4}}
\end{equation}
and $h(\bar{t})>0$. From this and $2<p<2+\frac{4}{N}$, we further find three positive constant  $R_1<T_\alpha<R_2$ such that $h(R_1)=h(R_2)=0$,
\[
h(t)>0\ \text{for}\ t\in\left(R_1,R_2\right),\quad h(t)<0\ \text{for}\ t\in\left(0,R_1\right)\cup\left(R_2,+\infty\right),\quad\text{and}\ h(T_\alpha)={\max}_{t\in\mathbb{R}^+}h(t)>0.
\]
\end{proof}

Now, let us define
\[
\mathcal{B}_{r,\alpha}=\Big\{u\in S_{r,\alpha}:\|\nabla u\|_2^2\le T_\alpha^2\Big\},
\]
from the Pincar\'{e} inequality
\[
\int_{\Omega}|\nabla u|^2\mathrm{d}x\ge\frac{\theta\alpha}{r^2}
\]
for any $u\in S_{r,\alpha}$, where $\theta$ is the principal eigenvalue of $-\Delta$ with Dirichlet boundary condition in $\Omega$. 
There holds $\mathcal{B}_{r,\alpha}=\emptyset$ for $r<\frac{\sqrt{\theta\alpha}}{T_\alpha}$, $\mathcal{B}_{r,\alpha}\neq\emptyset$ for $r\ge\frac{\sqrt{\theta\alpha}}{T_\alpha}$.

\begin{theorem}\label{th-b-e-r-lambda}
Under the assumptions of Theorem \ref{beta>0-e<0-Omega}, let us define
\begin{equation}\label{r-alpha}
  r_0:=\max\Bigg\{\frac{\sqrt{\theta\alpha}}{T_\alpha},
  \bigg(\frac{p\theta\big(1+\|V\|_{\frac{N}{2}}S^{-1}\big)}{2\beta}
  \alpha^{\frac{2-p}{2}}|\Omega|^{\frac{p-2}{2}}\bigg)^{\frac{2}{N(p-2)-4}}\Bigg\}.
  \end{equation}
Then  for $r>r_0$,
  \[
  e_{r,\alpha}:=\underset{u\in\mathcal{B}_{r,\alpha}}{\inf}E_{r}(u)<0
  \]
  is attained at $0<u_r\in\mathcal{B}_{r,\alpha}$. Moreover, there is $\lambda_r\in\mathbb{R}$ such that $(\lambda_r,u_r)$ is a solution of \eqref{main-eq-omega}, and  ${\liminf}_{r\to\infty}\lambda_r>0$.
\end{theorem}
\begin{proof}
Owing to
\[
\int_{\Omega}|\nabla v_1|^2\mathrm{d}x=\theta\alpha
\quad\text{and}\quad
\alpha=\int_{\Omega}|v_1|^2\mathrm{d}x
\le\bigg(\int_{\Omega}|v_1|^p\mathrm{d}x\bigg)^{\frac{2}{p}}|\Omega|^{\frac{p-2}{p}},
\]
let us take $u_0(x):=r^{-\frac{N}{2}}v_1(r^{-1}x)$ for $x\in\Omega_r$. Via direction computation, we conclude that $u_0\in\mathcal{B}_{r,\alpha}$,
\[
\int_{\Omega_r}|\nabla u_0|^2\mathrm{d}x=r^{-2}\theta\alpha
\quad\text{and}\quad
\int_{\Omega_r}|u_0|^p\mathrm{d}x\ge r^{\frac{N(2-p)}{2}}\alpha^{\frac{p}{2}}|\Omega|^{\frac{2-p}{2}}.
\]
By \eqref{r-alpha} and  $2<p<2+\frac{4}{N}$, we have
\begin{eqnarray*}
E_r(u_0)&=&\frac{1}{2}\int_{\Omega_r}|\nabla u_0|^2\mathrm{d}x
+\frac{1}{2}\int_{\Omega_r}Vu_0^2\mathrm{d}x
-\frac{1}{2^*}\int_{\Omega_r}|u_0|^{2^*}\mathrm{d}x
-\frac{\beta}{p}\int_{\Omega_r}|u_0|^{p}\mathrm{d}x\\
&\le&\frac{1}{2}\bigg(1+\|V\|_{\frac{N}{2}}S^{-1}\bigg)r^{-2}\theta\alpha
-\frac{\beta}{p}r^{\frac{N(2-p)}{2}}\alpha^{\frac{p}{2}}|\Omega|^{\frac{2-p}{2}}\\
&\le&0.
\end{eqnarray*}
Combining the definition of $e_{r,\alpha}$ leads to $e_{r,\alpha}<0$.
The Gagliardo-Nirenberg inequality implies that
\begin{eqnarray}\label{E-GN}
E_r(u_r)&\ge&\frac{1}{2}\bigg(1-\|V_-\|_{\frac{N}{2}S^{-1}}\bigg)\int_{\Omega}|\nabla u|^2\mathrm{d}x-\frac{\beta C_{N,p}}{p}
\alpha^{\frac{2p-N(p-2)}{4}}\bigg(\int_{\Omega}|\nabla u|^2\mathrm{d}x\bigg)^{\frac{N(p-2)}{4}}\notag\\
&&\quad-\frac{S^{-2^*/2}}{2^*}\bigg(\int_{\Omega}|\nabla u|^2\mathrm{d}x\bigg)^{\frac{2^*}{2}}.
\end{eqnarray}
This indicates that $E_r$ is bounded from below in $\mathcal{B}_{r,\alpha}$. Then from the Ekeland principle, there exists a sequence $\{u_{n,r}\}\subset\mathcal{B}_{r,\alpha}$ satisfying
\[
E_r(u_{n,r})\to e_{r,\alpha},\quad
E'_r(u_{n,r})|_{T_{u_n,r}S_{r,\alpha}}\to0\quad\text{as}\ n\to\infty.
\]
Namely, there exists $u_r\in H_0^1(\Omega_r)$ such that
\[
u_{n,r}\rightharpoonup u_r\ \text{in}\ H_0^1(\Omega_r)\quad u_{n,r}\to u_r\ \text{in}\ L^k(\Omega_r)\quad\text{for }2\le k<2^*.
\]
Moreover,
\[
\|\nabla u_r\|_2^2\le\underset{n\to\infty}{\liminf}\,\|\nabla u_{n,r}\|_2^2\le T_\alpha^2,
\]
that is $u_r\in\mathcal{B}_{r,\alpha}$.  By Willem \cite[Proposition 5.12]{Wil96}, there exists $\{\lambda_n\}$ such
that
\[
E'_{r}(u_{n,r})+\lambda_nu_{n,r}\to0 \quad\text{in}\ H^{-1}(\Omega_r).
\]
Besides,
\[
\int_{\Omega_r}Vu_{n,r}^2\mathrm{d}x\to\int_{\Omega_r}Vu_{r}^2\mathrm{d}x\quad\text{as}\ n\to\infty,
\]
we arrive that
\begin{equation}\label{e'unr>0}
o_n(1)=E'_{r}(u_{n,r})(u_{n,r}-u_r)
=\int_{\Omega_r}\big(|\nabla u_{n,r}|^2-|\nabla u_r|^2\big)\mathrm{d}x-\int_{\Omega_r}\big(|u_{n,r}|^{2^*}-|u_r|^{2^*}\big)\mathrm{d}x.
\end{equation}
In what follows, let us set $v_{n,r}=u_{n,r}-u_r$, we assume that going to a subsequence, if necessary,
\[
\int_{\Omega_r} |\nabla v_{n,r}|^2 \mathrm{d}x\to b\ge0\quad\text{as}\,n\to\infty.
\]
Thanks to Brezis-Lieb lemma and \eqref{e'unr>0}, we derive
\[
\int_{\Omega_r} |v_{n,r}|^{2^*} \mathrm{d}x\to b.
\]
By the definition of $S$,
\[
b\ge S\bigg(\int_{\Omega_r}|v_{n,r}|^{2^*}\mathrm{d}x\bigg)^{\frac{2}{2^*}}\ge Sb^\frac{N-2}{N}.
\]
Then $b=0$ or $b\ge S^{\frac{N}{2}}$.
Assume by contradiction that $b\ge S^{\frac{N}{2}}$, we know
\begin{eqnarray*}
e_{r,\alpha}+o_n(1)&=&\underset{n\to\infty}{\lim}E_{r}(u_{n,r})\\
&=&\underset{n\to\infty}{\lim}\Bigg\{\frac{1}{2}\int_{\Omega_r}|\nabla u_{n,r}|^2\mathrm{d}x
+\frac{1}{2}\int_{\Omega_r}V(x)u_{n,r}^2\mathrm{d}x
-\frac{1}{2^*}\int_{\Omega_r}|u_{n,r}|^{2^*}\mathrm{d}x
-\frac{\beta}{p}\int_{\Omega_r}|u_{n,r}|^p\mathrm{d}x\Bigg\}\\
&=&E_{r}(u_r)+\frac{1}{2}\int_{\Omega_r}\bigg(|\nabla u_{n,r}|^2-|\nabla u_r|^2\bigg)\mathrm{d}x
-\frac{1}{2^*}\int_{\Omega_r}\bigg(|u_{n,r}|^{2^*}-|u_r|^{2^*}\bigg)\mathrm{d}x\\
&\ge&e_{r,\alpha}+\frac{b}{N}\ge e_{r,\alpha}+S^{\frac{N}{2}}\frac{1}{N},
\end{eqnarray*}
which is not possible. Hence, $b=0$ and then $u_{n,r}\to u_r$ in $H_0^1(\Omega_r)$ as $n\to\infty$. Consequently, $E_{r}(u_r)=e_{r,\alpha}$ and  $u_r$ is an interior point of $\mathcal{B}_{r,\alpha}$ due to that  $E_r(u)>h(T_\alpha)>0$ for any $u\in\partial\mathcal{B}_{r,\alpha}$ by \eqref{E-GN}.

The Lagrange multiplier theorem implies that there exists $\lambda_r\in\mathbb{R}$ such that $(\lambda_r,u_r)$ is a solution of \eqref{main-eq-omega}. Recall that the definition of $e_{r,\alpha}$, we have  $e_{r_2,\alpha}\le e_{r_1,\alpha}<0$ for any $r_1>r_2>r_0$. In addition,
\begin{eqnarray}\label{eralpha-lambda}
\lambda_r\alpha&=&\int_{\Omega_r}|\nabla u_r|^2\mathrm{d}x
-\int_{\Omega_r}Vu_r^2\mathrm{d}x
+\int_{\Omega_r}|u_r|^{2^*}\mathrm{d}x
+\beta\int_{\Omega_r}|u_r|^{p}\mathrm{d}x\notag\\
&=&\frac{\beta(p-2)}{p}\int_{\Omega_r}|u_r|^{p}\mathrm{d}x
+\frac{2}{N}\int_{\Omega_r}|u_r|^{2^*}\mathrm{d}x-2E_r(u_r)\notag\\
&>&-2E_r(u_r)=-2e_{r,\alpha}.
\end{eqnarray}
Hence, ${\liminf}_{r\to\infty}\,\lambda_r>0$.
Moreover, the strong maximum principle implies $u_r>0$.
\end{proof}

The following is the priori bound to the solutions of problem \eqref{main-eq-omega}.
\begin{lemma}\label{ur-ubdd}
Let $\{(\lambda_r,u_r)\}$ be  nonnegative solutions of  \eqref{main-eq-omega}  with $\|u_r\|_{H^1}\le C$, where $C>0$ is independent of $r$, then $\limsup\limits_{r\to\infty}\,\|u_r\|_\infty<\infty$.
\end{lemma}
\begin{proof}
    We assume by the contradiction  that there are $\{u_r\}$ and $x_r\in\Omega_r$ such that $M_r:{\max}_{x\in\Omega_r}u_r(x)=u_r(x_r)\to\infty$ as $r\to\infty$. Set
    \[
\omega_r=\frac{u_r(x_r+\tau_r x)}{M_r}\quad\textbf{for}\ x\in\Sigma^r:=\{x\in\mathbb{R}^N:x_r+\tau_r x\in\Omega_r\},
    \]
    where $\tau_r=M_r^{-\frac{2}{N-2}}$. Through a delicate calculation, we get that $\tau_r\to0$, $\|\omega_r\|_{L^\infty(\Sigma^r)}\le1$ and $\omega_r$ satisfies
    \begin{align}\label{equ:241204-e1}
        -\Delta\omega_r+\tau_r^2V(x_r+\tau_rx)\omega_r+\tau^2_r\lambda_r\omega_r=|\omega_r|^{2^*-2}\omega_r+\beta\tau_r^{\frac{2(2^*-p)}{2^*-2}}|\omega_r|^{p-2}\omega_r\quad\textbf{in}\ \Sigma^r
    \end{align}

It follows from \eqref{main-eq-omega},  the Gagliardo-Nirenberg inequality,  the Sobolev inequatlity and $\left\|u_r\right\|_{H^1} \leq C$ that  the sequence $\left\{\lambda_r\right\}$ is bounded. According to   $L^p$ estimates  (see [20, Theorem 9.11]), we can deduce that $\omega_r \in W_{\text {loc }}^{2, p}(\Sigma)$ and $\|\omega_r\|_{W_{\text {loc }}^{2, p}(\Sigma)} \leq C$ for any $p>1$ and $\Sigma:=\lim\limits _{r \rightarrow \infty} \Sigma^r$. Using Sobolev embedding theorem, we can obtain that $\omega_r \in C_{\mathrm{loc}}^{1, \beta}(\Sigma)$ for some $\beta \in(0,1)$ and $\|\omega_r\|_{C_{\text {loc }}^{1, \beta}(\Sigma)} \leq C$.   Therefore,  applying  the Arzela-Ascoli theorem, we find that there exists $v$ such that up to a subsequence
 \[\omega_r\to\omega\quad\textbf{in}\  C_{loc}^\beta(\Sigma).\]
 where $\Sigma=\lim_{r\to\infty}\Sigma^r$ is a smooth domain.

  Next, we claim that    \[
\liminf_{r\to\infty}\frac{\text{dist}(x_r,\partial\Omega_r)}{\tau_r}>0.\]
  which follows from the standard direct method ( see \cite[Lemma2.7]{BQZ24}).  So we omit it here.
As a result, Let $r \to \infty$ in \eqref{equ:241204-e1}, we find that
   $\omega$ is a nonnegative solution of problem
    \[\begin{cases}
        -\Delta\omega =|\omega|^{2^*-2}\omega,&\text{in}\ \Sigma,\\
        \omega(x)=0&\text{on}\ \partial\Sigma.
    \end{cases}\]
 If \[
\liminf_{r\to\infty}\frac{\text{dist}(x_r,\partial\Omega_r)}{\tau_r}=\infty\]
occurs,  $\Sigma=\mathbb{R}^{N}$. Then from the Liouville theorem, there holds  $\omega=0$ in $\Sigma$, which is impossible due to the fact that $\omega(0)=\liminf_{r\to\infty}\omega_r(0)=1$. If \[
\liminf_{r\to\infty}\frac{\text{dist}(x_r,\partial\Omega_r)}{\tau_r}=d>0\]
occurs, we next \textbf{claim} $\Sigma$ is a half space. 
In fact, let  dist$(x_r,\partial\Omega_r)=|x_r-z_r|$ with $z_r\in\partial\Omega_r$, then $\tilde{z}_r=z_r/r\in\partial\Omega$ and $\tilde{x}_r=x_r/r\in\Omega$.  Based on the definition of $\Sigma_r$,  the origin is located at $\tilde{x}_r$ and $\tilde{x}_r-\tilde{z}_r=|x_r-z_r|(1,0,...,0)$ by translating  and rotating  the coordinate
system.
Assume up to a subsequence that $\tilde{z}_r\to z$ with $z\in\Omega$.
And by the smoothness of the domain see \cite[section 6.2]{GT83},  functions $f_r,f:\mathbb{R}^{N-1}\to\mathbb{R}$ are smooth  such that
\[f_r(0)=-|\tilde{x}_r-\tilde{z}_r|=-\frac{|x_r-z_r|}{r},\quad \nabla f_r(0)=0,\]
    \begin{equation}\label{>f}
        \Omega\cap B(\tilde{z}_r,\delta)=\{x\in B(\tilde{z}_r,\delta):x_1>f_r(x_2,...,x_N)\},
    \end{equation}
    and
    \begin{equation}\label{=f}
        \partial\Omega\cap B(\tilde{z}_r,\delta)=\{x\in B(\tilde{z}_r,\delta):x_1=f_r(x_2,...,x_N)\}.
    \end{equation}
    Moreover, \eqref{>f} and \eqref{=f} hold by replacing $\tilde{z}_r$ and $f_r$ with z and f , respectively.
And $f_r\to f$ in $C^1(B_{\frac{\delta}{2}})$, where $B_{\frac{\delta}{2}}\subset\mathbb{R}^{N-1}$. Consequently,
\[
 \partial\Omega_r\cap B(z_r,r\delta)=\{(rf_r(x'),rx'):|x'|<\delta\}=\big\{\bigg(rf_r(\frac{x'}{r}),x'\bigg):|x'|<r\delta\big\},
\]
where $x'\in\mathbb{R}^{N-1}$. As a consequence, for any $x'\in\mathbb{R}^{N-1}$ and large $r$,
\[
y_r=\bigg(rf_r(\frac{\tau_rx'}{r}),\tau_rx'\bigg)\in\partial\Omega_r\cap B(z_r,r\delta).
\]
Moreover,
\[
\frac{y_r-x_r}{\tau_r}=\bigg(\frac{rf_r(\frac{\tau_rx'}{r})}{\tau_r},x'\bigg).
\]
Note that
\begin{eqnarray*}
\frac{rf_r(\frac{\tau_rx'}{r})}{\tau_r}&=&
\frac{rf_r(\frac{\tau_rx'}{r})-rf_r(0)+rf_r(0)}{\tau_r}
=\frac{r\nabla f_r(\frac{\theta_r\tau_rx'}{r})\cdot\frac{\tau_rx'}{r}-|x_r-z_r|}{\tau_r}  \\
&\to&\nabla f(0)\cdot x'-d=-d.
\end{eqnarray*}
As a result, $\frac{y_r-x_r}{\tau_r}\to(-d,x')$ as $r\to\infty$. Therefore $\Sigma=\{x\in\mathbb{R}^{N}:x_1>-d\}$, the claim hold. Then we can use the Liouville theorem, $\omega=0$ in $\Sigma$, this contradicts $\omega(0)=\liminf_{r\to\infty}\omega_r(0)=1$.  The proof is now complete.
\end{proof}

\begin{remark}\label{ur-ubdd-beta}
Note that the proof of Lemma \ref{ur-ubdd}  does not depend on $\beta$.
\end{remark}

\noindent\textbf{Proof of Theorem  \ref{beta>0-e<0-Omega}:} The proof is an immediate consequence of Theorem \ref{th-b-e-r-lambda}, Lemma \ref{ur-ubdd} and Remark \ref{ur-ubdd-beta}.

\section{Proof of Theorem \ref{beta>0-e>0-Omega}}
\setcounter{equation}{0}
\setcounter{theorem}{0}	
This section is devoted to the existence of mountain-pass type solution for the cases $\beta>0$.
Under the assumptions of Theorem \ref{beta>0-e>0-Omega}, for  $s\in\left[1/2,1\right]$ we define the functional $J_{r,s}:S_{r,\alpha}\to\mathbb{R}$ by
\[
J_{r,s}(u)=\frac{1}{2}\int_{\Omega_r}|\nabla u|^2\mathrm{d}x
+\frac{1}{2}\int_{\Omega_r}Vu^2\mathrm{d}x
-s\,\bigg(\frac{1}{2^*}\int_{\Omega_r}|u|^{2^*}\mathrm{d}x
+\frac{\beta}{p}\int_{\Omega_r}|u|^{p}\mathrm{d}x\bigg).
\]
Note that if $u\in S_{r,\alpha}$ is a critical point of $J_{r,s}$, then there is $\lambda\in\mathbb{R}$ such that $(\lambda,u)$ is a solution of the problem
\begin{equation}\label{main-eq-s-omega-2}
\begin{cases}
-\Delta u+V(x)u+\lambda u=s|u|^{2^*-2}u+s\beta|u|^{p-2}u&\text{in}\ \Omega_r,\\
u\in H^{1}_{0}(\Omega_r),\ \int_{\Omega_r}|u|^2dx=\alpha.
\end{cases}
\end{equation}

\begin{lemma}\label{mp-betage0}
For any $0<\alpha<\tilde{\alpha}_V$, where $\tilde{\alpha}_V$ is defined in Theorem \ref{beta>0-e>0-Omega}, there exists $\tilde{r}_\alpha>0$ and $u^0,u^1\in S_{\tilde{r}_\alpha,\alpha}$ such that
\begin{description}
  \item[(i)] $J_{r,s}(u^1)\le0$ for any $r>\tilde{r}_\alpha$ and $s\in\left[\frac{1}{2},1\right]$,
  \[
  \|\nabla u^0\|_2^2<\bigg(\frac{(1-\|V_-\|_{\frac{N}{2}}S^{-1})}{A}
  \bigg)^{\frac{N-2}{2}}<\|\nabla u^1\|_2^2
  \]
  and
  \[
  J_{r,s}(u^0)<
  \frac{1}{N}\bigg(\frac{1}{A}\bigg)^{\frac{N-2}{2}}
\bigg(1-\|V_-\|_{\frac{N}{2}}S^{-1}\bigg)^{\frac{N}{2}},
  \]
  where
  \[
  A=\bigg(\frac{32S^{-2^*/2^*}}{(N-2)(p-2)(4-N(p-2))}+S^{-2^*/2}\bigg)
  \]
  \item[(ii)] If $u\in S_{r,\alpha}$ satisfies
  \[
  \|\nabla u\|_2^2=\bigg(\frac{(1-\|V_-\|_{\frac{N}{2}}S^{-1})}{A}
  \bigg)^{\frac{N-2}{2}},
  \]
  then there holds
  \[
  J_{r,s}(u)\ge\frac{1}{N}\bigg(\frac{1}{A}\bigg)^{\frac{N-2}{2}}
\bigg(1-\|V_-\|_{\frac{N}{2}}S^{-1}\bigg)^{\frac{N}{2}}.
  \]
\end{description}
\end{lemma}

\begin{proof}
 Clearly the set $S_{r,\alpha}$ is path connected.  Let
\[
v_\varepsilon(x):=\frac{\sqrt{\alpha}}{\|u_\varepsilon\|_2}U_\varepsilon(x).
\]
We can easily verify $v_\varepsilon(x)\in C_0^{\infty}(\Omega)$ and $v_\varepsilon(x)\in S_1^\alpha$.
 Setting $v_t(x)=t^{\frac{N}{2}}v_\varepsilon(tx)$ for $x\in\Omega_{\frac{1}{t}}$ and $t>0$ there holds
\begin{eqnarray}\label{j-vt-s-h}
J_{\frac{1}{t},s}(v_t)&\le&\frac{1}{2}t^2\bigg(1+\|V\|_{\frac{N}{2}}S^{-1}\bigg)\int_{\Omega}|\nabla v_\varepsilon|^2\mathrm{d}x-\frac{\beta}{2p}t^{\frac{N(p-2)}{2}}\int_{\Omega}|v_\varepsilon|^p\mathrm{d}x\notag\\
&&\quad-\frac{1}{22^*}t^{2^*}\int_{\Omega}|v_\varepsilon|^{2^*}\mathrm{d}x\\
&\le& h(t)\notag,
\end{eqnarray}
where  $h:\mathbb{R}^+\to\mathbb{R}$ is defined by
\[
h(t):=\frac{1}{2}\bigg(1+\|V\|_{\frac{N}{2}}S^{-1}\bigg)t^2
\int_{\Omega}|\nabla v_\varepsilon|^2\mathrm{d}x
-\frac{1}{22^*}t^{2^*}\int_{\Omega}|v_\varepsilon|^{2^*}\mathrm{d}x.
\]
Via direct computation, there hold $h(t_0)=0$ with
\[
t_0:=\bigg(2^*+2^*\|V\|_{\frac{N}{2}}S^{-1}\bigg)^{\frac{N-2}{4}},
\]
 $h(t)<0$ for   $t>t_0$ and $h(t)>0$ for   $0<t<t_0$.
Then we obtain that for any $r\ge\frac{1}{t_0}$ and $s\in\left[\frac{1}{2},1\right]$,
\begin{equation}\label{j-rs-le0}
J_{r,s}(v_{t_0})=J_{\frac{1}{t_0},s}(v_{t_0})\le h(t_0)=0.
\end{equation}
Additional the function $h$ attains its maximum at
\[
t_\alpha=\bigg(2+2\|V\|_{\frac{N}{2}}S^{-1}\bigg)^{\frac{N-2}{4}},
\]
and so we find  $t_1\in\left(0,t_\alpha\right)$ such that for any $t\in\left[0,t_1\right]$,
\begin{equation}\label{j-h-bdd}
h(t)<h(t_1)\le\frac{1}{N}\bigg(\frac{1}{A}\bigg)^{\frac{N-2}{2}}
\bigg(1-\|V_-\|_{\frac{N}{2}}S^{-1}\bigg)^{\frac{N}{2}}.
\end{equation}
On the other hand, the Sobolev inequality and the Gagliardo-Nirenberg inequality implies that
\begin{eqnarray}\label{jrs-bel}
J_{r,s}(u)&\ge&\frac{1}{2}\bigg(1-\|V_-\|_{\frac{N}{2}}S^{-1}\bigg)\int_{\Omega_r}|\nabla u|^2dx
-\frac{C_{N,p}\beta\alpha^{\frac{2p-N(p-2)}{4}}}{p}\bigg(\int_{\Omega_r}|\nabla u|^2dx\bigg)^{\frac{N(p-2)}{4}}\notag\\
&&\quad-\frac{S^{-\frac{2^*}{2}}}{2^*}\bigg(\int_{\Omega_r}|\nabla u|^2\mathrm{d}x\bigg)^{\frac{2^*}{2}}.
\end{eqnarray}
Let $f:\mathbb{R}^+\to\mathbb{R}$ be defined by
\[
f(t):=\frac{1}{2}\bigg(1-\|V_-\|_{\frac{N}{2}}S^{-1}\bigg)t
-\frac{C_{N,p}\beta\alpha^{\frac{2p-N(p-2)}{4}}}{p}t^{\frac{N(p-2)}{4}}
-\frac{S^{-\frac{2^*}{2}}}{2^*}t^{\frac{2^*}{2}}.
\]
Then there exist three positive constants $l_1<l_M<l_2$ such that
\[
f(t)<0\ \ \text{for}\ t\in\left(0,l_1\right)\cup\left(l_2,\infty\right),\quad\ f(t)>0\ \ \text{for}\ t\in\left(l_1,l_2\right),
\quad\text{and}\ f(l_M)=\max_{t\in\mathbb{R}^+}f(t)>0.
\]
After a direct calculation, we deduce $f''(t)\le0$ if and only if $t>t_2$ with
\[
t_2=\bigg(\frac{C_{N,p}\beta}{p}\frac{2^*}{A_pS^{-2^*/2}}
\alpha^{\frac{2p-N(p-2)}{4}}\bigg)^{\frac{4}{22^*-N(p-2)}},
\]
and we further deduce
\[
\underset{t\in\mathbb{R}^+}{\max}f(t)=\underset{t\in\left[t_2,\infty\right)}{\max}f(t).
\]
For any $t\ge t_2$, it holds
\begin{eqnarray}\label{jrs-bel}
f(t)&=&\frac{1}{2}\bigg(1-\|V_-\|_{\frac{N}{2}}S^{-1}\bigg)t
-\frac{S^{-\frac{2^*}{2}}}{2^*}t^{\frac{2^*}{2}}
-\frac{C_{N,p}\beta\alpha^{\frac{2p-N(p-2)}{4}}}{p}t^{\frac{N(p-2)}{4}}\notag\\
&\ge&\frac{1}{2}\bigg(1-\|V_-\|_{\frac{N}{2}}S^{-1}\bigg)t
-\frac{1}{2^*}\bigg(S^{-2^*/2^*}A_p+S^{-2^*/2}\bigg)t^{\frac{2^*}{2}}\notag\\
&=:&g(t).
\end{eqnarray}
Let
\[
A=\bigg(S^{-2^*/2^*}A_p+S^{-2^*/2}\bigg),
\quad A_p=\frac{32}{(N-2)(p-2)(4-N(p-2))}
\]
and
\[
t_g=\bigg(\frac{1-\|V_-\|_{\frac{N}{2}}S^{-1}}{A}
  \bigg)^{\frac{N-2}{2}},
\]
so that $t_g>t_2$ by the definition of $\tilde{\alpha}_V$,  $\max_{t\in\left[t_2,\infty\right)}g(t)=g(t_g)$ and
\[
\underset{t\in\mathbb{R}^+}{\max}f(t)
=\underset{t\in\left[t_2,\infty\right)}{\max}g(t)
=\frac{1}{N}\bigg(\frac{1}{A}\bigg)^{\frac{N-2}{2}}
\bigg(1-\|V_-\|_{\frac{N}{2}}S^{-1}\bigg)^{\frac{N}{2}}.
\]
Selecting $\bar{r}_\alpha=\max\{\frac{1}{t_1},(\frac{2S^{\frac{N}{2}}+10}{t_g})^{\frac{1}{2}}\}$, we derive that  $v_{\frac{1}{\bar{r}_\alpha}}\in S_{\bar{r}_\alpha,\alpha}\subset S_{r,\alpha}$ for $r>\bar{r}_{\alpha}$, and
\begin{equation}\label{nabla-v-ralpha-bdd}
\|\nabla v_{\frac{1}{\bar{r}_\alpha}}\|_2^2
=\bigg(\frac{1}{\bar{r}_\alpha}\bigg)^2\|\nabla v_1\|_2^2
<\bigg(\frac{1-\|V_-\|_{\frac{N}{2}}S^{-1}}{A}
  \bigg)^{\frac{N-2}{2}},
\end{equation}
moreover,
\begin{equation}\label{j-h-t1-bdd}
J_{\bar{r}_\alpha,s}\bigg(v_{\frac{1}{\bar{r}_\alpha}}\bigg)\le h\bigg(\frac{1}{{\bar{r}_\alpha}}\bigg)
\le h(t_1).
\end{equation}
Now we define $u^0:=v_{\frac{1}{\bar{r}_\alpha}}$, $u^1:=v_{t_0}$ and
\[
\tilde{r}_\alpha:=\max\Big\{\frac{1}{t_0},\bar{r}_\alpha\Big\}.
\]
Thanks to \eqref{j-rs-le0}-\eqref{j-h-bdd} and \eqref{nabla-v-ralpha-bdd}-\eqref{j-h-t1-bdd}, we derive (i) holds. Furthermore, by employing \eqref{jrs-bel}, we also infer (ii) holds.
\end{proof}

\begin{lemma}\label{mp-level-betage0}
Under the assumptions of Theorem \ref{beta>0-e>0-Omega}, set
  \[
  m_{r,s}(\alpha)=\underset{\gamma\in\Gamma_{r,\alpha}}{\inf}
  \underset{t\in\left[0,1\right]}{\sup}J_{r,s}(\gamma(t)),
  \]
  with
  \[
  \Gamma_{r,\alpha}=\Big\{\gamma\in C(\left[0,1\right],S_{r,\alpha}):\gamma(0)=u^0,\gamma(1)=u^1\Big\}.
  \]
  Then
  \[
  \frac{1}{N}\bigg(\frac{1}{A}\bigg)^{\frac{N-2}{2}}
\bigg(1-\|V_-\|_{\frac{N}{2}}S^{-1}\bigg)^{\frac{N}{2}}\le m_{r,s}(\alpha)\le  \frac{1}{N}s^{\frac{2-N}{2}}S^{\frac{N}{2}}
  \]
\end{lemma}
\begin{proof}
 Since $J_{r,s}(u^1)\le0$ for any $\gamma\in\Gamma_{r,\alpha}$, we have
\[
\|\nabla \gamma(0)\|_2^2<t_g<\|\nabla\gamma(1)\|_2^2.
\]
It then follows from \eqref{jrs-bel} that
\[
\underset{t\in\left[0,1\right]}{\max}J_{r,s}(\gamma(t))\ge g(t_g)=\frac{1}{N}\bigg(\frac{1}{A}\bigg)^{\frac{N-2}{2}}
\bigg(1-\|V_-\|_{\frac{N}{2}}S^{-1}\bigg)^{\frac{N}{2}}
\]
for any $\gamma\in\Gamma_{r,\alpha}$, hence the first inequality in Lemma \ref{mp-level-betage0} holds. Now we define a path $\gamma:\left[0,1\right]\to S_{r,\alpha}$ by
\[
\gamma(\tau):\Omega_r\to\mathbb{R},\quad x\mapsto\bigg(\tau t_0+(1-\tau)\frac{1}{\tilde{r}_\alpha}\bigg)^{\frac{N}{2}}v_\varepsilon\bigg(\bigg(\tau t_0+(1-\tau)\frac{1}{\tilde{r}_\alpha}\bigg)x\bigg).
\]
Clearly $\gamma\in\Gamma_{r,\alpha}$.
Then by the definition of $m_{r,s}(\alpha)$, we have
\begin{eqnarray*}
m_{r,s}(\alpha)&=&\underset{\gamma\in\Gamma_{r,\alpha}}{\inf}\,
\underset{t\in\left[0,1\right]}{\sup}\,J_{r,s}(\gamma(t))
\le\underset{t\in\left[0,1\right]}{\sup}J_{r,s}(\gamma_0(t))\le\underset{t\ge0}{\max}J_{r,s}(v_t)=\underset{t\ge0}{\max}J_{\frac{1}{t},s}(v_t)\\
&\le&\underset{t\ge0}{\max}\Bigg\{\frac{1}{2}t^2\int_{\Omega}|\nabla v_\varepsilon|^2\mathrm{d}x
-\frac{\beta}{p}t^{\frac{N(p-2)}{2}}\int_{\Omega}|v_\varepsilon|^p\mathrm{d}x
-\frac{s}{2^*}t^{2^*}\int_{\Omega}|v_\varepsilon|^{2^*}\mathrm{d}x+\frac{1}{2}\int_{\Omega}V(x)v_\varepsilon^2\mathrm{d}x\Bigg\}\\
&\le&\underset{t\ge0}{\max}g(t)+O(\varepsilon^{N-2}),
\end{eqnarray*}
where
\[
g(t):=\frac{1}{2}t^2\int_{\Omega}|\nabla v_\varepsilon|^2\mathrm{d}x-\frac{s}{2^*}t^{2^*}\int_{\Omega}|v_\varepsilon|^{2^*}\mathrm{d}x.
\]
In view of the estimates of Lemma \ref{u-varepsilon-es} (i), it holds that
\[
\underset{t\ge0}{\max}g(t)=s^{\frac{2-N}{2}}\frac{1}{N}S^{\frac{N}{2}}+O(\varepsilon^{N-2}).
\]
Summarizing, we have $m_{r,s}(\alpha)<\frac{1}{N}S^{\frac{N}{2}}$. The proof of Lemma \ref{mp-level-betage0} is now complete.
\end{proof}

\begin{theorem}\label{beta>0-s-omega-ge}
Let $r>\tilde{r}_\alpha$, where $\tilde{r}_\alpha$ is defined in Lemma \ref{mp-betage0}. There is $\alpha_1>0$ such that problem \eqref{main-eq-s-omega-2} admits a solution $(\lambda_{r,s},u_{r,s})$ for almost every $s\in\left[\frac{1}{2},1\right]$ and $0<\alpha<\alpha_1$. Moreover, there hold $u_{r,s}>0$ and $J_{r,s}(u_{r,s})=m_{r,s}(\alpha)$.
\end{theorem}
\begin{proof}
For fixed $\alpha\in\left(0,\tilde{\alpha}_V\right)$, let us apply Theorem \ref{mp-th} to $E_{r,s}$ with $\Gamma_{r,\alpha}$ given in Lemma \ref{mp-level-betage0},
\[
A(u)=\frac{1}{2}\int_{\Omega}|\nabla u|^2\mathrm{d}x+\frac{1}{2}\int_{\Omega}V(x)u^2\mathrm{d}x
-\frac{\beta}{p}\int_{\Omega}|u|^p\mathrm{d}x
\]
and
\[
B(u)=\frac{1}{2^*}\int_{\Omega}|u|^{2^*}\mathrm{d}x.
\]
Thanks to Lemma \ref{mp-betage0}, the assumptions in Theorem \ref{mp-th} hold. Hence, for almost every $s\in\left[\frac{1}{2},1\right]$, there exists a nonnegative bounded Palais-Smale sequence $\{u_n\}$:
\[
J_{r,s}(u_n)\to m_{r,s}(\alpha)\quad\text{and}\quad J'_{r,s}(u_n)|_{T_{u_n}S_{r,\alpha}}\to0,
\]
where $T_{u_n}S_{r,\alpha}$ denoted the tangent space of $S_{r,\alpha}$ at $u_n$.
Note that
\begin{equation}\label{un-ers-ge}
J'_{r,s}(u_n)+\lambda_nu_n\to0 \quad\text{in}\ H^{-1}(\Omega_r)
\end{equation}
and
\[
\lambda_n=-\frac{1}{\alpha}\bigg(\int_{\Omega_r}|\nabla u_n|^2\mathrm{d}x+\int_{\Omega_r}V(x) u^2_n\mathrm{d}x-\beta\int_{\Omega_r}|u_n|^p\mathrm{d}x-s\int_{\Omega_r}|u_n|^{2^*}\mathrm{d}x\bigg).
\]
is bounded, that is  $\lambda_n\le K$, where constant $K>0$.
Furthermore, there exist $u_0\in H_0^1(\Omega_r)$ and $\lambda\in\mathbb{R}$ such that up to a subsequence,
\[
\lambda_n\to\lambda,\quad u_n\rightharpoonup u_0\ \text{in}\ H_0^1(\Omega_r)\quad \text{and}\ u_n\to u_0\ \text{in}\ L^t(\Omega_r)\ \text{for all }2\le t<2^*,
\]
and $u_0$ satisfies
\begin{equation}\label{eq-u0-ge}
\begin{cases}
-\Delta u_0+Vu_0+\lambda u_0=s|u_0|^{2^*-2}u_0+\beta|u_0|^{p-2}u_0&\text{in}\,\Omega_r,\\
u_0\in H_0^1(\Omega_r),\quad\int_{\Omega_r}|u_0|^2\mathrm{d}x=\alpha.
\end{cases}
\end{equation}
We first claim that there exists a positive constant $C$ such that  $\int_{\Omega_r}|\nabla u_0|^2\mathrm{d}x\ge C$. Indeed, from \eqref{eq-u0-ge} and Sobolev embedding inequality, we have
\begin{eqnarray*}
(1-\|V_-\|_{\frac{N}{2}}S^{-1})\int_{\Omega_r}|\nabla u_0|^2\mathrm{d}x
&\le&\int_{\Omega_r}|\nabla u_0|^2\mathrm{d}x+\int_{\Omega_r}V(x)u_0^2\mathrm{d}x+\lambda_n\alpha\\
&=&s\int_{\Omega_r}|u_0|^{2^*}\mathrm{d}x+\beta\int_{\Omega_r}|u_0|^{p}\mathrm{d}x\\
&\le&s\bigg(S^{-1}\int_{\Omega_r}|\nabla u_0|^2\mathrm{d}x\bigg)^{\frac{2^*}{2}}
+C'\beta\bigg(\int_{\Omega_r}|\nabla u_0|^2\mathrm{d}x\bigg)^{\frac{p}{2}}.
\end{eqnarray*}
This proves the claim. Let $\alpha_0:=\frac{ C(p-2)}{2K}(1-\|V_-\|_{\frac{N}{2}}S^{-1})$, we next claim that $J_{r,s}(u_0)\ge0$ for $0<\alpha\le\alpha_0$.  In fact, from \eqref{un-ers-ge} and \eqref{eq-u0-ge}, we know that
\begin{eqnarray*}
J_{r,s}(u_0)&=&J_{r,s}(u_0)-\frac{1}{p}(J'_{r,s}(u_0)u_0+\lambda_n\alpha)\\
&\ge&(\frac{1}{2}-\frac{1}{p})\int_{\Omega_r}|\nabla u_0|^2\mathrm{d}x
+(\frac{1}{2}-\frac{1}{p})\int_{\Omega_r}V(x)u_0^2\mathrm{d}x
+(\frac{1}{p}-\frac{1}{2^*})s\int_{\Omega_r}|u_0|^{2^*}\mathrm{d}x-\frac{1}{p}\lambda_n\alpha\\
&\ge&(\frac{1}{2}-\frac{1}{p})(1-\|V_-\|_{\frac{N}{2}}S^{-1})\int_{\Omega_r}|\nabla u_0|^2\mathrm{d}x-\frac{1}{p}K\alpha\\
&\ge&(\frac{1}{2}-\frac{1}{p})(1-\|V_-\|_{\frac{N}{2}}S^{-1})C-\frac{1}{p}K\alpha_0=0.
\end{eqnarray*}
In view of \eqref{un-ers-ge}, we have
\[
J'_{r,s}(u_n)u_0+\lambda_n\int_{\Omega_r}u_nu_0\mathrm{d}x\to0\quad\text{and}\quad J'_{r,s}(u_n)u_n+\lambda_n\alpha\to0\quad\text{as}\quad n\to\infty.
\]
Owing to
\[
\underset{n\to\infty}{\lim}\int_{\Omega_r}V(x)u_n^2\mathrm{d}x=\int_{\Omega_r}V(x)u_0^2\mathrm{d}x.
\]
We conclude that
\begin{equation}\label{e'un0-ge}
o_n(1)=J'_{r,s}(u_n)(u_n-u_0)
=\int_{\Omega_r}\big(|\nabla u_n|^2-|\nabla u_0|^2\big)\mathrm{d}x-s\int_{\Omega_r}\big(|u_n|^{2^*}-|u_0|^{2^*}\big)\mathrm{d}x.
\end{equation}
Next, let us define $v_n:=u_n-u_0$, going to a subsequence, we assume that
\[
\int_{\Omega_r} |\nabla v_n|^2 \mathrm{d}x\to b\ge0\quad\text{as}\,n\to\infty.
\]
Note that  Brezis-Lieb Lemma and \eqref{e'un0-ge}, so we arrive that
\[
\int_{\Omega_r} |v_n|^{2^*} \mathrm{d}x\to \frac{b}{s}\quad\text{as}\,n\to\infty.
\]
The definition of $S$ indicates that
\[
\int_{\Omega_r} |\nabla v_n|^2 \mathrm{d}x\ge S\bigg(\int_{\Omega_r}|v_n|^{2^*}\mathrm{d}x\bigg)^{\frac{2}{2^*}},
\]
and moreover, $b=0$ or $b\ge s^{\frac{2-N}{2}}S^{\frac{N}{2}}$.
We argue by contradiction that $b\ge s^{\frac{2-N}{2}}S^{\frac{N}{2}}$,
\begin{eqnarray*}
m_{r,s}(\alpha)+o_n(1)&=&\underset{n\to\infty}{\lim}J_{r,s}(u_n)\\
&=&\underset{n\to\infty}{\lim}\Bigg\{\frac{1}{2}\int_{\Omega_r}|\nabla u_n|^2\mathrm{d}x
+\frac{1}{2}\int_{\Omega_r}V(x)u_n^2\mathrm{d}x
-\frac{s}{2^*}\int_{\Omega_r}|u_n|^{2^*}\mathrm{d}x
-\frac{\beta}{p}\int_{\Omega_r}|u_n|^p\mathrm{d}x\Bigg\}\\
&=&J_{r,s}(u_0)+\frac{1}{2}\int_{\Omega_r}\bigg(|\nabla u_n|^2-|\nabla u_0|^2\bigg)\mathrm{d}x
-\frac{s}{2^*}\int_{\Omega_r}\bigg(|u_n|^{2^*}-|u_0|^{2^*}\bigg)\mathrm{d}x\\
&\ge&\frac{b}{N}\ge s^{\frac{2-N}{2}}S^{\frac{N}{2}}\frac{1}{N}.
\end{eqnarray*}
However, this is not possible since $m_{r,s}(\alpha)<\frac{1}{N}s^{\frac{2-N}{2}}S^{\frac{N}{2}}$. Hence, $b=0$ and then $u_n\to u_0$ in $H_0^1(\Omega_r)$ as $n\to\infty$. Consequently, $J_{r,s}(u_0)=m_{r,s}(\alpha)$ and $u_0$ is a nonnegative normalized solution to \eqref{main-eq-s-omega}. Defining $\alpha_1:=\min\{\tilde{\alpha}_V,\alpha_0\}$, the proof is now complete.
\end{proof}

\begin{lemma}\label{nabla-u-bdd-j}
For fixed $\alpha>0$ the set of solutions $u\in S_{r,\alpha}$ of \eqref{main-eq-s-omega-2} is bounded uniformly in $s$ and $r$.
\end{lemma}
\begin{proof}
Based on the fact that $(\lambda,u)\in\mathbb{R}\times S_{r,\alpha}$ is a solution of problem \eqref{main-eq-s-omega-2}, we know
\begin{equation}\label{eq-sol-bdd-ge}
\int_{\Omega_r}|\nabla u|^2\mathrm{d}x+\int_{\Omega_r}V(x)u^2\mathrm{d}x
+\lambda\int_{\Omega_r}|u|^2\mathrm{d}x=s\int_{\Omega_r}|u|^{2^*}\mathrm{d}x
+s\beta\int_{\Omega_r}|u|^p\mathrm{d}x
.
\end{equation}
In addition, the pohozaev identity leads to
\begin{eqnarray}\label{pohozaev-ge}
&&\frac{N-2}{2N}\int_{\Omega_r}|\nabla u|^2\mathrm{d}x+\frac{1}{2N}\int_{\partial\Omega_r}|\nabla u|^2(x\cdot\textbf{n})\mathrm{d}\sigma+\frac{1}{2N}\int_{\Omega_r}\tilde{V}u^2\mathrm{d}x
+\frac{1}{2}\int_{\Omega_r}Vu^2\mathrm{d}x\notag\\
&&\quad\quad=-\frac{\lambda}{2}\int_{\Omega_r}|u|^2\mathrm{d}x
+\frac{s}{2^*}\int_{\Omega_r}|u|^{2^*}\mathrm{d}x
+\frac{s\beta}{p}\int_{\Omega_r}|u|^p\mathrm{d}x,
\end{eqnarray}
where $\textbf{n}$ denotes the outward unit normal vector on $\partial\Omega_r$.
Then via the inequality \eqref{eq-sol-bdd-ge} and \eqref{pohozaev-ge}, we deduce that
\begin{eqnarray*}
\frac{1}{N}\int_{\Omega_r}|\nabla u|^2\mathrm{d}x&-&\frac{1}{2N}\int_{\partial\Omega_r}|\nabla u|^2(x\cdot\textbf{n})\mathrm{d}\sigma-\frac{1}{2N}\int_{\Omega_r}(\nabla V\cdot x)u^2\mathrm{d}x\\
&=&\frac{s}{N}\int_{\Omega_r}|u|^{2^*}\mathrm{d}x+\frac{\beta(p-2)s}{2p}\int_{\Omega_r}|u|^p\mathrm{d}x\\
&\ge&\frac{2^*}{N}\bigg(\frac{s}{2^*}\int_{\Omega_r}|u|^{2^*}\mathrm{d}x
+\frac{\beta s}{p}\int_{\Omega_r}|u|^p\mathrm{d}x\bigg)
+\frac{\beta s(p-2^*)}{2p}\int_{\Omega_r}|u|^p\mathrm{d}x\\
&=&\frac{2^*}{N}\bigg(\frac{1}{2}\int_{\Omega_r}|\nabla u|^2\mathrm{d}x
+\frac{1}{2}\int_{\Omega_r}V|u|^2\mathrm{d}x-m_{r,s(\alpha)}\bigg)+\frac{\beta s(p-2^*)}{2p}\int_{\Omega_r}|u|^p\mathrm{d}x,
\end{eqnarray*}
where we have used $\beta>0$.  Recall that  $\Omega_r$ is starshaped with respect to 0, so $x\cdot\textbf{n}\ge0$ for any $x\in\partial\Omega_r$. Thereby, applying the Gagliardo-Nirenberg inequality, we obtain
\begin{eqnarray*}
&&\frac{2^*}{N}m_{r,s}(\alpha)+\frac{\beta s(2^*-p)C_{N,p}}{2p}\alpha^{\frac{2p-N(p-2)}{4}}
\bigg(\int_{\Omega}|\nabla u|^2\mathrm{d}x\bigg)^{\frac{N(p-2)}{4}}\\
&\ge&\frac{2^*}{N}m_{r,s}(\alpha)+\frac{\beta s(2^*-p)}{2p}\int_{\Omega_r}|u|^p\mathrm{d}x\\
&=&\frac{2^*}{N}\bigg(\frac{1}{2}\int_{\Omega_r}|\nabla u|^2\mathrm{d}x
+\frac{1}{2}\int_{\Omega_r}V|u|^2\mathrm{d}x\bigg)\\
&&\quad-\bigg(\frac{1}{N}\int_{\Omega_r}|\nabla u|^2\mathrm{d}x
-\frac{1}{2N}\int_{\partial\Omega_r}|\nabla u|^2(x\cdot\textbf{n})\mathrm{d}\sigma-\frac{1}{2N}\int_{\Omega_r}(\nabla V\cdot x)u^2\mathrm{d}x\bigg)\\
&\ge&2^*\int_{\Omega_r}|\nabla u|^2\mathrm{d}x-\alpha\bigg(\frac{1}{2N}\|\nabla V\cdot x\|_\infty+\frac{1}{N-2}\|V\|_\infty\bigg).
\end{eqnarray*}
As a consequence of $2<p<2+\frac{4}{N}$ and  Lemma \ref{mp-betage0}, we could bound $\int_{\Omega_r}|\nabla u|^2\mathrm{d}x$ uniformly in $s$ and $r$.
\end{proof}

\begin{lemma}\label{beta>0-e-strong-ge}
Assume that $0<\alpha<\alpha_1$ and  $r>\tilde{r}_\alpha$, where $\alpha_1$ and $\tilde{r}_\alpha$ are given in Theorem \ref{beta>0-s-omega-ge} and  Lemma \ref{mp-betage0}, respectively. Equation \eqref{main-eq-omega} admits a solution $(\lambda_{r,\alpha},u_{r,\alpha})$ such that $u_{r,\alpha}>0$ in $\Omega_r$.
\end{lemma}
\begin{proof}
 For fixed $0<\alpha<\alpha_1$ and  $r>\tilde{r}_\alpha$. Based on the preceding Theorem \ref{beta>0-s-omega-ge} and Lemma \ref{nabla-u-bdd-j}, we demonstrate that there exist solutions $\{(\lambda_{r,\alpha,s},u_{r,\alpha,s})\}$ to problem \eqref{main-eq-s-omega-2} for $s\in\left[1/2,1\right]$, and $\{u_{r,\alpha,s}\}\subset S_{r,\alpha}$ is bounded. Now, using a technique similar to the proof of  Theorem \ref{beta>0-s-omega-ge}, we obtain $u_{r,\alpha}\in S_{r,\alpha}$ and $\lambda_{r,\alpha}\in\mathbb{R}$ such that going to a subsequence, $\lambda_{r,\alpha,s}\to\lambda_{r,\alpha}$ and $u_{r,\alpha,s}\to u_{r,\alpha}$ in $H_0^{1}(\Omega_r)$ as $s\to1$. This combined with the strong maximum principle leads to that $u_{r,\alpha}>0$ is a solution of problem \eqref{main-eq-omega}.
\end{proof}

\noindent\textbf{Proof of Theorem \ref{beta>0-e>0-Omega}} The proof is an immediate consequence of Theorem \ref{beta>0-s-omega-ge}, Lemma \ref{beta>0-e-strong-ge} and Remark \ref{ur-ubdd-beta}.\qed

\section{Proof of Theorem \ref{betale0-e>0-Omega}}
\setcounter{equation}{0}
\setcounter{theorem}{0}	
In this section, we always assume that the assumptions of Theorem \ref{betale0-e>0-Omega} hold. In
order to obtain a bounded Palais-Smale sequence, we will use the monotonicity trick.
For $\frac{1}{2}\le s\le1$, the energy functional $E_{r,s}:S_{r,\alpha}\to\mathbb{R}$ is defined by
\[
E_{r,s}(u)=\frac{1}{2}\int_{\Omega_r}|\nabla u|^2dx+\frac{1}{2}\int_{\Omega_r}V(x)u^2dx
-\frac{s}{2^*}\int_{\Omega_r}|u|^{2^*}dx-\frac{\beta}{p}\int_{\Omega_r}|u|^pdx.
\]
Note that  if $u\in S_{r,\alpha}$ is a critical point of $E_{r,\alpha}$, then there exists $\lambda\in\mathbb{R}$ such that $(\lambda,u)$ is a solution of the equation
\begin{equation}\label{main-eq-s-omega}
\begin{cases}
-\Delta u+V(x)u+\lambda u=s|u|^{2^*-2}u+\beta|u|^{p-2}u&\text{in}\ \Omega_r,\\
u\in H^{1}_{0}(\Omega_r),\ \int_{\Omega_r}|u|^2dx=\alpha.
\end{cases}
\end{equation}

\begin{lemma}\label{betale0-mp-g}
For any $\alpha>0$, there exists $r_\alpha>0$ and $u^0,u^1\in S_{r,\alpha}$ such that
\begin{description}
  \item[(i)] $E_{r,s}(u^1)\le0$ for any $r>r_\alpha$ and $s\in\left[\frac{1}{2},1\right]$,
  \[
  \|\nabla u^0\|_2^2<\bigg(1-\|V_-\|_{\frac{N}{2}}S^{-1}\bigg)^{\frac{N-2}{2}}S^{\frac{N}{2}}<\|\nabla u^1\|_2^2
  \]
  and
  \[
  E_{r,s}(u^0)<\frac{1}{N}\bigg(1-\|V_-\|_{\frac{N}{2}}S^{-1}\bigg)^{\frac{N}{2}}S^{\frac{N}{2}}
  \]
  \item[(ii)] If $u\in S_{r,\alpha}$ satisfies
  \[
  \|\nabla u\|_2^2=\bigg(1-\|V_-\|_{\frac{N}{2}}S^{-1}\bigg)^{\frac{N-2}{2}}S^{\frac{N}{2}}
  \]
  then there holds
  \[
  E_{r,s}(u)\ge\frac{1}{N}\bigg(1-\|V_-\|_{\frac{N}{2}}S^{-1}\bigg)^{\frac{N}{2}}S^{\frac{N}{2}}
  \]
\end{description}
\end{lemma}
\begin{proof}
 Clearly the set $S_{r,\alpha}$ is path connected. Let
\[
v_\varepsilon(x):=\frac{\sqrt{\alpha}}{\|u_\varepsilon\|_2}u_\varepsilon(x).
\]
Notice that $v_\varepsilon(x)\in C_0^{\infty}(\Omega)$ and $v_\varepsilon(x)\in S_1^\alpha$.
 Setting $v_t(x)=t^{\frac{N}{2}}v_\varepsilon(tx)$ for $x\in\Omega_{\frac{1}{t}}$ and $t>0$ there holds
\begin{eqnarray}\label{e-vt-s-h}
E_{\frac{1}{t},s}(v_t)&\le&\frac{1}{2}t^2\bigg(1+\|V\|_{\frac{N}{2}}S^{-1}\bigg)\int_{\Omega}|\nabla v_\varepsilon|^2\mathrm{d}x-\frac{\beta C_{N,p}}{p}\alpha^{\frac{2p-N(p-2)}{4}}\bigg(t^2\int_{\Omega}|\nabla v_\varepsilon|^2\mathrm{d}x\bigg)^{\frac{N(p-2)}{4}}\notag\\
&&\quad-\frac{1}{22^*}t^{2^*}\int_{\Omega}|v_\varepsilon|^{2^*}\mathrm{d}x\\
&\le& h(t)\notag,
\end{eqnarray}
Note that since $2<p<2+\frac{4}{N}$ and $\beta<0$, there exist $0<T_\alpha<t_\alpha$ such that $h(t_\alpha)=0$, $h(t)<0$ for any $t>t_\alpha$, $h(t)>0$ for any $0<t<t_\alpha$ and $h(T_\alpha)={\max}_{t\in\mathbb{R}^+} h(t)$. As a consequence, there holds
\begin{equation}\label{htalpha0}
E_{r,s}(v_{t_\alpha})=E_{\frac{1}{t_{\alpha}},s}(v_{t_\alpha})\le h(t_\alpha)=0
\end{equation}
for any $r\ge\frac{1}{t_\alpha}$ and $s\in\left[\frac{1}{2},1\right]$. Moreover, there exists $0<t_1<T_\alpha$ such that for any $t\in\left[0,t_1\right)$,
\begin{equation}\label{htupbdd}
h(t)<h(t_1)\le\frac{1}{N}\bigg(1-\|V_-\|_{\frac{N}{2}}S^{-1}\bigg)^{\frac{N}{2}}S^{\frac{N}{2}}.
\end{equation}

On the other hand, it follows from the Gagliardo-Nirenberg inequality and the H\"{o}lder inequality that
\begin{equation}\label{ers-bel}
E_{r,s}(u)\ge\frac{1}{2}\bigg(1-\|V_-\|_{\frac{N}{2}}S^{-1}\bigg)\int_{\Omega_r}|\nabla u|^2dx-\frac{S^{-\frac{2^*}{2}}}{2^*}\bigg(\int_{\Omega_r}|\nabla u|^2dx\bigg)^{\frac{2^*}{2}}.
\end{equation}
Let
\[
f(t):=\frac{1}{2}\bigg(1-\|V_-\|_{\frac{N}{2}}S^{-1}\bigg)t-\frac{S^{-\frac{2^*}{2}}}{2^*}
t^{\frac{2^*}{2}}
\]
and
\[
\tilde{t}:=\bigg(1-\|V_-\|_{\frac{N}{2}}S^{-1}\bigg)^{\frac{N-2}{2}}S^{\frac{N}{2}}.
\]
Then $f$ is increasing in $(0,\tilde{t})$ and decreasing in $(\tilde{t},\infty)$, and
\[
f(\tilde{t})=\frac{1}{N}\bigg(1-\|V_-\|_{\frac{N}{2}}S^{-1}\bigg)^{\frac{N}{2}}S^{\frac{N}{2}}.
\]
For $r>\tilde{r}_\alpha:=
\max\Big\{\frac{1}{t_1},(\frac{2
S^{\frac{N}{2}}+10)}{\tilde{t}})^{\frac{1}{2}}\Big\}$, we have $v_{\frac{1}{\tilde{r}_\alpha}}\in S_{\tilde{r}_\alpha,\alpha}\subset S_{r,\alpha}$ and
\begin{equation}\label{nabla-u0}
\|\nabla v_{\frac{1}{\tilde{r}_\alpha}}\|_2^2=\bigg(\frac{1}{\tilde{r}_\alpha}\bigg)^2\|\nabla v_\varepsilon\|_2^2<\bigg(1-\|V_-\|_{\frac{N}{2}}S^{-1}\bigg)^{\frac{N-2}{2}}S^{\frac{N}{2}}.
\end{equation}
Moreover, there holds
\begin{equation}\label{e-ralpha-s-ht1}
E_{\tilde{r}_\alpha,s}(v_{\frac{1}{\tilde{r}_\alpha}})\le h\bigg(\frac{1}{\tilde{r}_\alpha}\bigg)<h(t_1).
\end{equation}
Setting $u^0=v_{\frac{1}{\tilde{r}_\alpha}}$, $u^1=v_{t_\alpha}$ and
\[
r_\alpha=\max\Big\{\frac{1}{t_\alpha},\tilde{r}_\alpha\Big\},
\]
the statement (i) holds due to  \eqref{htalpha0}-\eqref{htupbdd} and \eqref{nabla-u0}-\eqref{e-ralpha-s-ht1}.
(ii) Note that statement (ii) holds by \eqref{ers-bel} and a direct calculation.
\end{proof}

\begin{lemma}\label{mp-level}
 Suppose that the assumptions in Lemma \ref{betale0-mp-g} are in force. Set
  \[
  m_{r,s}(\alpha)=\underset{\gamma\in\Gamma_{r,\alpha}}{\inf}\,
  \underset{t\in\left[0,1\right]}{\sup}\,E_{r,s}(\gamma(t)),
  \]
  with
  \[
  \Gamma_{r,\alpha}=\Big\{\gamma\in C(\left[0,1\right],S_{r,\alpha}):\gamma(0)=u^0,\gamma(1)=u^1\Big\}.
  \]
  Then
  \[
  \frac{1}{N}\bigg(1-\|V_-\|_{\frac{N}{2}}S^{-1}\bigg)^{\frac{N}{2}}S^{\frac{N}{2}}\le m_{r,s}(\alpha)<\frac{1}{N}s^{\frac{2-N}{2}}S^{\frac{N}{2}}.
  \]
\end{lemma}
\begin{proof}
  Since $E_{r,s}(u^1)\le0$ for any $\gamma\in\Gamma_{r,\alpha}$, we have
\[
\|\nabla \gamma(0)\|_2^2<\tilde{t}=\bigg(1-\|V_-\|_{\frac{N}{2}}S^{-1}\bigg)^{\frac{N-2}{2}}S^{\frac{N}{2}}<\|\nabla\gamma(1)\|_2^2.
\]
It then follows from \eqref{ers-bel} that
\[
\underset{t\in\left[0,1\right]}{\max}E_{r,s}(\gamma(t))\ge f(\tilde{t})=\frac{1}{N}\bigg(1-\|V_-\|_{\frac{N}{2}}S^{-1}\bigg)^{\frac{N}{2}}S^{\frac{N}{2}}
\]
for any $\gamma\in\Gamma_{r,\alpha}$, hence the first inequality  holds. For the second inequality,  we define a path $\gamma_0\in\Gamma_{r,\alpha}$ by
\[
\gamma_0(\tau):\Omega_r\to\mathbb{R},\quad x\mapsto\bigg(\tau t_\alpha+(1-\tau)\frac{1}{\tilde{r}_\alpha}\bigg)^{\frac{N}{2}}v_\varepsilon\bigg(\bigg(\tau t_\alpha+(1-\tau)\frac{1}{\tilde{r}_\alpha}\bigg)x\bigg)
\]
for $\tau\in\left[0,1\right]$. Then by the definition of $m_{r,s}(\alpha)$, we have
\begin{eqnarray*}
m_{r,s}(\alpha)&=&\underset{\gamma\in\Gamma_{r,\alpha}}{\inf}\,
\underset{t\in\left[0,1\right]}{\sup}\,E_{r,s}(\gamma(t))
\le\underset{t\in\left[0,1\right]}{\sup}E_{r,s}(\gamma_0(t))\\
&\le&\underset{t\ge0}{\max}E_{r,s}(v_t)=\underset{t\ge0}{\max}E_{\frac{1}{t},s}(v_t)\\
&\le&\underset{t\ge0}{\max}\Bigg\{\frac{1}{2}t^2\int_{\Omega}|\nabla v_\varepsilon|^2\mathrm{d}x+\frac{1}{2}\int_{\Omega}V(x)v_\varepsilon^2\mathrm{d}x
-\frac{\beta}{p}t^{\frac{N(p-2)}{2}}\int_{\Omega}|v_\varepsilon|^p\mathrm{d}x
-\frac{s}{2^*}t^{2^*}\int_{\Omega}|v_\varepsilon|^{2^*}\mathrm{d}x\Bigg\}\\
&\le&\underset{t\ge0}{\max}\Bigg\{g(t)+\frac{1}{2}\int_{\Omega}V(x)v_\varepsilon^2\mathrm{d}x
-\frac{\beta}{p}t^{\frac{N(p-2)}{2}}\int_{\Omega}|v_\varepsilon|^p\mathrm{d}x\Bigg\}\\
&\le&\underset{t\ge0}{\max}g(t)+\frac{1}{2}\int_{\Omega}V(x)v_\varepsilon^2\mathrm{d}x
-\frac{\beta}{p}t_\varepsilon^{\frac{N(p-2)}{2}}\int_{\Omega}|v_\varepsilon|^p\mathrm{d}x,
\end{eqnarray*}
where
\[
g(t):=\frac{1}{2}t^2\int_{\Omega}|\nabla v_\varepsilon|^2\mathrm{d}x-\frac{s}{2^*}t^{2^*}\int_{\Omega}|v_\varepsilon|^{2^*}\mathrm{d}x,
\]
\[
H(t):=g(t)+\frac{1}{2}\int_{\Omega}V(x)v_\varepsilon^2\mathrm{d}x
-\frac{\beta}{p}t^{\frac{N(p-2)}{2}}\int_{\Omega}|v_\varepsilon|^p\mathrm{d}x
\]
and $t_\varepsilon>0$ satisfy ${\max}_{t\ge0}H(t)=H(t_\varepsilon)$. We next \textbf{claim} that there exist $0<T_0<T_1$ such that $t_\varepsilon\in\left[T_0,T_1\right]$. Indeed, by a simple calculation,
\begin{eqnarray*}
H'(t)&=&t\int_{\Omega}|\nabla v_\varepsilon|^2\mathrm{d}x
-\frac{N\beta(p-2)}{2p}t^{\frac{N(p-2)}{2}-1}\int_{\Omega}|v_\varepsilon|^p\mathrm{d}x
-st^{2^*-1}\int_{\Omega}|v_\varepsilon|^{2^*}\mathrm{d}x\\
&=&t^{\frac{N(p-2)}{2}-1}\bigg(t^{2-\frac{N(p-2)}{2}}\int_{\Omega}|\nabla v_\varepsilon|^2\mathrm{d}x
-\frac{N\beta(p-2)}{2p}\int_{\Omega}|v_\varepsilon|^p\mathrm{d}x
-st^{2^*-\frac{N(p-2)}{2}}\int_{\Omega}|v_\varepsilon|^{2^*}\mathrm{d}x\bigg),
\end{eqnarray*}
and then $t_\varepsilon$ satisfy
\[
t_\varepsilon^{2-\frac{N(p-2)}{2}}\int_{\Omega}|\nabla v_\varepsilon|^2\mathrm{d}x
-\frac{N\beta(p-2)}{2p}\int_{\Omega}|v_\varepsilon|^p\mathrm{d}x
=st_\varepsilon^{2^*-\frac{N(p-2)}{2}}\int_{\Omega}|v_\varepsilon|^{2^*}\mathrm{d}x.
\]
It then follows from $\beta\le0$ and Young inequality that there exist $0<T_0<T_1$ such that $t_\varepsilon\in\left[T_0,T_1\right]$. Hence, by the definition of $v_\varepsilon$, we conclude that
\[
m_{r,s}(\alpha)\le\underset{t\ge0}{\max}g(t)+\frac{1}{2}\int_{\Omega}V(x)v_\varepsilon^2\mathrm{d}x
-\frac{\beta}{p}T_1^{\frac{N(p-2)}{2}}\int_{\Omega}|v_\varepsilon|^p\mathrm{d}x.
\]
In view of the estimates of Lemma \ref{u-varepsilon-es} (i), it holds that
\[
\underset{t\ge0}{\max}g(t)=s^{\frac{2-N}{2}}\frac{1}{N}S^{\frac{N}{2}}+O(\varepsilon^{N-2}).
\]
Summarizing, the proof of (iii) will be completed if we manage to show that, for $n\in\mathbb{N}$ large enough,
\[
\frac{1}{2}\int_{\Omega}V(x)v_{\varepsilon_n}^2\mathrm{d}x
-\frac{\beta}{p}T_1^{\frac{N(p-2)}{2}}\int_{\Omega}|v_{\varepsilon_n}|^p\mathrm{d}x
+O(\varepsilon_n^{N-2})<0.
\]
At this point we distinguish two cases.

\textbf{Case 1 $N\ge5$:} By Lemma \ref{u-varepsilon-es} we have that, for some $K_1>0$, $K_2>0$, as $n\to\infty$,
\[
\|v_{\varepsilon_n}\|_2^2=K_1\varepsilon_n^2+o(\varepsilon_n^{N-2})\quad\text{and}\quad
\|v_{\varepsilon_n}\|_q^2=K_2\varepsilon_n^{N-\frac{(N-2)p}{2}}+o(\varepsilon_n^{N-\frac{(N-2)p}{2}}).
\]
Thus, for some constants, $\tilde{K}_1>0$, $\tilde{K}_2>0$, for $n\in\mathbb{N}$ sufficiently large,
\[
\frac{1}{2}\int_{\Omega}V(x)v_{\varepsilon_n}^2\mathrm{d}x
-\frac{\beta}{p}T_1^{\frac{N(p-2)}{2}}\int_{\Omega}|v_{\varepsilon_n}|^p\mathrm{d}x
+O(\varepsilon_n^{N-2})
=\tilde{K}_1\varepsilon_n^2+\tilde{K}_2\varepsilon_n^{N-\frac{(N-2)p}{2}}+O(\varepsilon_n^{N-2})<0
\]
since $N-\frac{(N-2)p}{2}<\min\{N-2,2\}\Longleftrightarrow p>2$.

\textbf{Case 2 $N=4$:} By Lemma \ref{u-varepsilon-es} we have that, for some $K_4>0$, as $n\to\infty$,
\[
\|v_{\varepsilon_n}\|_2^2=\omega\varepsilon_n^2|\text{ln}\varepsilon_n|+o(\varepsilon_n^{2})\quad\text{and}\quad
\|v_{\varepsilon_n}\|_q^2=K_4\varepsilon_n^{4-p}+o(\varepsilon_n^{4-p}).
\]
Thus, for some constants, $\tilde{K}_3>0$, $\tilde{K}_4>0$, for $n\in\mathbb{N}$ sufficiently large,
\[
\frac{1}{2}\int_{\Omega}V(x)v_{\varepsilon_n}^2\mathrm{d}x
-\frac{\beta}{p}T_1^{\frac{N(p-2)}{2}}\int_{\Omega}|v_{\varepsilon_n}|^p\mathrm{d}x
+O(\varepsilon_n^{N-2})
=\tilde{K}_3\varepsilon_n^2|\text{ln}\varepsilon_n|
-\tilde{K}_4\varepsilon_n^{4-p}+O(\varepsilon_n^{2})
\le-\frac{\tilde{K}_4}{2}\varepsilon_n^{4-p}<0
\]
since $4-p<2$. In view of Cases 1 and 2 we finish the proof of Lemma \ref{mp-level} if $N\ge4$.
\end{proof}

In view of Lemma \ref{betale0-mp-g}, the energy functional $E_{r,s}$ possesses the mountain pass geometry.  For the sequence obtained from  Theorem \ref{mp-th}, we will show its convergence in the next theorem.
\begin{theorem}\label{beta>0-s-omega}
For $r>r_\alpha$, where $r_\alpha$ is defined in Lemma \ref{betale0-mp-g}. There is $\alpha_2>0$ such that problem \eqref{main-eq-s-omega} admits a solution $(\lambda_{r,s},u_{r,s})$ for almost every $s\in\left[\frac{1}{2},1\right]$ and $0<\alpha<\alpha_2$. Moreover, $u_{r,s}\ge0$ and $E_{r,s}(u_{r,s})=m_{r,s}$.
\end{theorem}
\begin{proof}
 Let us apply Theorem \ref{mp-th} to $E_{r,s}$ with $\Gamma_{r,\alpha}$ given in Lemma \ref{mp-level},
\[
A(u)=\frac{1}{2}\int_{\Omega}|\nabla u|^2\mathrm{d}x+\frac{1}{2}\int_{\Omega}V(x)u^2\mathrm{d}x
-\frac{\beta}{p}\int_{\Omega}|u|^p\mathrm{d}x
\]
and
\[
B(u)=\frac{1}{2^*}\int_{\Omega}|u|^{2^*}\mathrm{d}x.
\]
Note that the assumptions in Theorem \ref{mp-th} hold due to $\beta\le0$ and Lemma \ref{betale0-mp-g}. Hence, for almost every $s\in\left[\frac{1}{2},1\right]$, there exists a nonnegative bounded Palais-Smale sequence $\{u_n\}$:
\[
E_{r,s}(u_n)\to m_{r,s}(\alpha)\quad\text{and}\quad E'_{r,s}(u_n)|_{T_{u_n}S_{r,\alpha}}\to0
\]
where $T_{u_n}S_{r,\alpha}$ denoted the tangent space of $S_{r,\alpha}$ at $u_n$. 
Note that
\[
\lambda_n=-\frac{1}{\alpha}\bigg(\int_{\Omega_r}|\nabla u_n|^2\mathrm{d}x+\int_{\Omega_r}V(x) u^2_n\mathrm{d}x-\beta\int_{\Omega_r}|u_n|^p\mathrm{d}x-s\int_{\Omega_r}|u_n|^{2^*}\mathrm{d}x\bigg)
\le K
\]
is bounded and
\begin{equation}\label{un-ers}
E'_{r,s}(u_n)+\lambda_nu_n\to0 \quad\text{in}\ H^{-1}(\Omega_r).
\end{equation}
As a consequence there exist $u_0\in H_0^1(\Omega_r)$ and $\lambda\in\mathbb{R}$ such that up to a subsequence,
\[
\lambda_n\to\lambda,\quad u_n\rightharpoonup u_0\ \text{in}\ H_0^1(\Omega_r)\quad \text{and}\ u_n\to u_0\ \text{in}\ L^t(\Omega_r)\ \text{for all }2\le t<2^*,
\]
and $u_0$ satisfies
\begin{equation}\label{eq-u0}
\begin{cases}
-\Delta u_0+Vu_0+\lambda u_0=s|u_0|^{2^*-2}u_0+\beta|u_0|^{p-2}u_0&\text{in}\,\Omega_r,\\
u_0\in H_0^1(\Omega_r),\quad\int_{\Omega_r}|u_0|^2\mathrm{d}x=\alpha.
\end{cases}
\end{equation}
We first \textbf{claim} that there exists a positive constant $C$ such that  $\int_{\Omega_r}|\nabla u_0|^2\mathrm{d}x\ge C$. Indeed, from \eqref{eq-u0} and Sobolev embedding inequality, we have
\begin{eqnarray*}
(1-\|V_-\|_{\frac{N}{2}}S^{-1})\int_{\Omega_r}|\nabla u_0|^2\mathrm{d}x
&\le&\int_{\Omega_r}|\nabla u_0|^2\mathrm{d}x+\int_{\Omega_r}V(x)u_0^2\mathrm{d}x+\lambda_n\alpha\\
&=&s\int_{\Omega_r}|u_0|^{2^*}\mathrm{d}x+\beta\int_{\Omega_r}|u_0|^{p}\mathrm{d}x\\
&\le&Cs\bigg(\int_{\Omega_r}|\nabla u_0|^2\mathrm{d}x\bigg)^{\frac{2^*}{2}},
\end{eqnarray*}
this implies the claim. Let $\alpha_2:=\frac{ C(p-2)}{2K}(1-\|V_-\|_{\frac{N}{2}}S^{-1})$, we next claim that $E_{r,s}(u_0)\ge0$ for $0<\alpha\le\alpha_2$.  In fact, from \eqref{un-ers} and \eqref{eq-u0}, we know that
\begin{eqnarray*}
E_{r,s}(u_0)&=&E_{r,s}(u_0)-\frac{1}{p}(E'_{r,s}(u_0)u_0+\lambda_n\alpha)\\
&\ge&(\frac{1}{2}-\frac{1}{p})\int_{\Omega_r}|\nabla u_0|^2\mathrm{d}x
+(\frac{1}{2}-\frac{1}{p})\int_{\Omega_r}V(x)u_0^2\mathrm{d}x
+(\frac{1}{p}-\frac{1}{2^*})s\int_{\Omega_r}|u_0|^{2^*}\mathrm{d}x-\frac{1}{p}\lambda_n\alpha\\
&\ge&(\frac{1}{2}-\frac{1}{p})(1-\|V_-\|_{\frac{N}{2}}S^{-1})\int_{\Omega_r}|\nabla u_0|^2\mathrm{d}x-\frac{1}{p}K\alpha\\
&\ge&(\frac{1}{2}-\frac{1}{p})(1-\|V_-\|_{\frac{N}{2}}S^{-1})C-\frac{1}{p}K\alpha_0=0.
\end{eqnarray*}
In view of \eqref{un-ers}, we have
\[
E'_{r,s}(u_n)u_0+\lambda_n\int_{\Omega_r}u_nu_0\mathrm{d}x\to0\quad\text{and}\quad E'_{r,s}(u_n)u_n+\lambda_n\alpha\to0\quad\text{as}\quad n\to\infty.
\]
Owing to
\[
\underset{n\to\infty}{\lim}\int_{\Omega_r}V(x)u_n^2\mathrm{d}x=\int_{\Omega_r}V(x)u_0^2\mathrm{d}x.
\]
We conclude that
\begin{equation}\label{e'un0}
o_n(1)=E'_{r,s}(u_n)(u_n-u_0)
=\int_{\Omega_r}\big(|\nabla u_n|^2-|\nabla u_0|^2\big)\mathrm{d}x-s\int_{\Omega_r}\big(|u_n|^{2^*}-|u_0|^{2^*}\big)\mathrm{d}x.
\end{equation}
Next, let us define $v_n:=u_n-u_0$, going to a subsequence, we assume that
\[
\int_{\Omega_r} |\nabla v_n|^2 \mathrm{d}x\to b\ge0\quad\text{as}\,n\to\infty.
\]
Note that  Brezis-Lieb Lemma and \eqref{e'un0}, so we arrive that
\[
\int_{\Omega_r} |v_n|^{2^*} \mathrm{d}x\to \frac{b}{s}\quad\text{as}\,n\to\infty.
\]
The definition of $S$ indicates that
\[
\int_{\Omega_r} |\nabla v_n|^2 \mathrm{d}x\ge S\bigg(\int_{\Omega_r}|v_n|^{2^*}\mathrm{d}x\bigg)^{\frac{2}{2^*}},
\]
and moreover, $b=0$ or $b\ge s^{\frac{2-N}{2}}S^{\frac{N}{2}}$.
We argue by contradiction that $b\ge s^{\frac{2-N}{2}}S^{\frac{N}{2}}$,
\begin{eqnarray*}
m_{r,s}(\alpha)+o_n(1)&=&\underset{n\to\infty}{\lim}E_{r,s}(u_n)\\
&=&\underset{n\to\infty}{\lim}\Bigg\{\frac{1}{2}\int_{\Omega_r}|\nabla u_n|^2\mathrm{d}x
+\frac{1}{2}\int_{\Omega_r}V(x)u_n^2\mathrm{d}x
-\frac{s}{2^*}\int_{\Omega_r}|u_n|^{2^*}\mathrm{d}x
-\frac{\beta}{p}\int_{\Omega_r}|u_n|^p\mathrm{d}x\Bigg\}\\
&=&E_{r,s}(u_0)+\frac{1}{2}\int_{\Omega_r}\bigg(|\nabla u_n|^2-|\nabla u_0|^2\bigg)\mathrm{d}x
-\frac{s}{2^*}\int_{\Omega_r}\bigg(|u_n|^{2^*}-|u_0|^{2^*}\bigg)\mathrm{d}x\\
&\ge&\frac{b}{N}\ge s^{\frac{2-N}{2}}S^{\frac{N}{2}}\frac{1}{N}.
\end{eqnarray*}
However, this is not possible since $m_{r,s}(\alpha)<\frac{1}{N}s^{\frac{2-N}{2}}S^{\frac{N}{2}}$. Hence, $b=0$ and then $u_n\to u_0$ in $H_0^1(\Omega_r)$ as $n\to\infty$. Consequently, $E_{r,s}(u_0)=m_{r,s}(\alpha)$ and $u_0$ is a nonnegative normalized solution to \eqref{main-eq-s-omega}. The proof is complete.
\end{proof}

Next, we will further establish a uniform estimate  in order to get a solution of problem \eqref{main-eq-omega}.
\begin{lemma}\label{nabla-u-bdd}
Suppose that $(\lambda,u)\in\mathbb{R}\times S_{r,\alpha}$ is a solution of problem \eqref{main-eq-s-omega} established in Theorem \ref{beta>0-s-omega} for some $r$ and $s$, then
\[
\int_{\Omega_r}|\nabla u|^2\mathrm{d}x\le \frac{1}{2}(2S)^{\frac{N}{2}}+\alpha\bigg(\frac{N-2}{4}\|\nabla V\cdot x\|_\infty+\frac{N}{2}\|V\|_\infty\bigg).
\]
\end{lemma}
\begin{proof}
 By an argument similar to that in Lemma \ref{nabla-u-bdd-j}, we deduce that
\begin{eqnarray*}
\frac{1}{N}\int_{\Omega_r}|\nabla u|^2\mathrm{d}x&-&\frac{1}{2N}\int_{\partial\Omega_r}|\nabla u|^2(x\cdot\textbf{n})\mathrm{d}\sigma-\frac{1}{2N}\int_{\Omega_r}(\nabla V\cdot x)u^2\mathrm{d}x\\
&=&\frac{s}{N}\int_{\Omega_r}|u|^{2^*}\mathrm{d}x+\frac{\beta(p-2)}{2p}\int_{\Omega_r}|u|^p\mathrm{d}x\\
&\ge&\frac{2^*}{N}\bigg(\frac{s}{2^*}\int_{\Omega_r}|u|^{2^*}\mathrm{d}x
+\frac{\beta}{p}\int_{\Omega_r}|u|^p\mathrm{d}x\bigg)\\
&=&\frac{2^*}{N}\bigg(\frac{1}{2}\int_{\Omega_r}|\nabla u|^2\mathrm{d}x
+\frac{1}{2}\int_{\Omega_r}V|u|^2\mathrm{d}x-m_{r,s(\alpha)}\bigg),
\end{eqnarray*}
where $\beta\le0$ have used. Recall that  $\Omega_r$ is starshaped with respect to 0, so $x\cdot\textbf{n}\ge0$ for any $x\in\partial\Omega_r$, thereby,
\begin{eqnarray*}
\frac{2^*}{N}m_{r,s}(\alpha)&\ge&\frac{2^*}{N}\bigg(\frac{1}{2}\int_{\Omega_r}|\nabla u|^2\mathrm{d}x
+\frac{1}{2}\int_{\Omega_r}V|u|^2\mathrm{d}x\bigg)\\
&&\quad-\bigg(\frac{1}{N}\int_{\Omega_r}|\nabla u|^2\mathrm{d}x
-\frac{1}{2N}\int_{\partial\Omega_r}|\nabla u|^2(x\cdot\textbf{n})\mathrm{d}\sigma-\frac{1}{2N}\int_{\Omega_r}(\nabla V\cdot x)u^2\mathrm{d}x\bigg)\\
&\ge&\frac{2^*-2}{2N}\int_{\Omega_r}|\nabla u|^2\mathrm{d}x-\alpha\bigg(\frac{1}{2N}\|\nabla V\cdot x\|_\infty+\frac{2^*}{2N}\|V\|_\infty\bigg).
\end{eqnarray*}
According to  $m_{r,s}(\alpha)<\frac{1}{2N}(2S)^{\frac{N}{2}}$, the proof of Lemma \ref{nabla-u-bdd} is now complete.
\end{proof}

\begin{lemma}\label{lambda-r>0-betale0}
For every $\alpha\in\left(0,\alpha_2\right)$, where $\alpha_2$ is given in Theorem \ref{beta>0-s-omega}, problem \eqref{main-eq-s-omega}  has a solution $(\lambda_r,u_r)$ provided $r>r_\alpha$ where $r_\alpha$ is as in Lemma \ref{betale0-mp-g}. Moreover, $u_r>0$ in $\Omega_r$.
\end{lemma}
\begin{proof}
 The proof  is similar to that of Lemma \ref{beta>0-e-strong-ge}, so we omit it here.
\end{proof}

\noindent\textbf{Acknowledgement}\ \ 

\bibliographystyle{amsplain}

\end{document}